\def\draft{n}
\documentclass[11pt]{amsart}
\usepackage[headings]{fullpage}
\usepackage{amssymb,epic,eepic,epsfig,amsbsy,amsmath,amscd,color}
\usepackage[all]{xy}
\usepackage{graphicx}
\usepackage{texdraw}
\usepackage{url}
\usepackage{bbm}
\usepackage{mathrsfs}
\usepackage{tikz}
\usetikzlibrary{cd}
\usepackage{accents}
\usepackage[normalem]{ulem}

\def\printname#1{
        \if\draft y
                \smash{\makebox[0pt]{\hspace{-0.5in}
                        \raisebox{8pt}{\tt\tiny #1}}}
        \fi
}

\def\lbl#1{\label{#1}\printname{#1}}

\def\biblbl{\bibitem}

                        \theoremstyle{plain}
\newtheorem{theorem}{Theorem}[section]

\newtheorem{lemma}[theorem]{Lemma}
\newtheorem{corollary}[theorem]{Corollary}
\newtheorem{proposition}[theorem]{Proposition}

\theoremstyle{definition}

\newtheorem{remark}[theorem]{Remark}

      \def\nc{\newcommand}

   \nc\FI[2]{\begin{figure}
    \begin{center}\input{#1.pstex_t}\end{center}
    \caption{#2}
    \lbl{#1}
  \end{figure}}
\nc\FIG[3]{\begin{figure}
    \includegraphics[#3]{#1.eps}
    \caption{#2}
    \lbl{fig:#1}
    \end{figure}}
\nc\FF[3]{\begin{figure}
    \includegraphics[#3]{#1.eps}
    \caption{#2}
    \lbl{#1}
    \end{figure}}
    \nc\FIGc[3]{\begin{figure}[htpb]
    \includegraphics[height=#3]{#1.eps}
    \caption{#2}
    \lbl{fig:#1}
    \end{figure}}

    \nc\FIGh[3]{\begin{figure}[htpb]
    \includegraphics[height=#3]{#1.eps}
    \caption{#2}
    \lbl{fig:#1}
    \end{figure}}

\newcommand\incl[2]{{\includegraphics[height=#1]{#2.eps}}}

\def\BA{\mathbb A}
\def\BC{\mathbb C}
\def\BN{\mathbb N}
\def\BZ{\mathbb Z}

\def\BQ{\mathbb Q}

\def\CZ{\mathcal Z}

\def\la{\langle}
\def\ra{\rangle}

\def\al{\alpha}
\def\ve{\varepsilon}

\def\be { \begin{equation} }
\def\ee { \end{equation} }

\newcommand\no[1]{}

\def\hD{{\hat D}}

\def\hatP{{\widehat{\partial}}}

\def\bT{\mathbb T}

\def\tS{\tilde \Sigma}

\def\cS{\mathscr S}

\def\Mat{\mathrm{Mat}}
\def\bk{\mathbf k}

\def\bl{\mathbf l}
\def\bn{\mathbf n}

\def\oP{{\mathring P}}

\def\cP{\mathcal P}
\def\tD{\D \cup \Hu}

\def\Id{\mathrm{Id}}

\def\D{\Delta}

\def\embed{\hookrightarrow}

\def\sX{\mathfrak X}
\def\tX{\tilde {\sX}}
\def\inn{{\mathrm{in}}}
\def\fm{\mathfrak m}

\def\cH{\mathcal H}
\def\cT{\mathcal T}

\def\Col{\mathrm{Col}}

\def\embed{\hookrightarrow}

\def\tX{\tilde{\cX}}

\def\vp{\varphi}

\def\SM{(\Sigma,\cP)}
\def\pS{\partial \Sigma}

\def\id{\mathrm{id}}
\def\Dess{\D_{\mathrm{ess}}}

\def\pM{\partial M}
\def\cN{\mathcal N}
\def\MN{(M,\cN)}
\def\tiX{\widetilde {\sX}}
\def\tcS{\widetilde \cS}
\def\tF{\tilde F}
\def\tPhi{\widetilde{\Phi}}
\def\XD{\sX(\D)}
\def\vpD{\vp_\D}
\def\vpDp{\vp_{\D'}}
\def\Hu{{\cH}}
\def\Hm{\cH_\bullet}
\def\ZD{\CZ(\D)}
\def\LMN{\cT\MN}
\def\fr{\operatorname{fr}}

\begin{document}

\title[Skein algebra]{On Kauffman bracket skein modules of marked 3-manifolds and the Chebyshev-Frobenius homomorphism}

\author[Thang  T. Q. L\^e]{Thang  T. Q. L\^e}
\address{School of Mathematics, 686 Cherry Street,
 Georgia Tech, Atlanta, GA 30332, USA}
\email{letu@math.gatech.edu}
\author[Jonathan Paprocki]{Jonathan Paprocki}
\address{School of Mathematics, 686 Cherry Street,
 Georgia Tech, Atlanta, GA 30332, USA}
\email{jon.paprocki@gatech.edu}


\thanks{
2010 {\em Mathematics Classification:} Primary 57N10. Secondary 57M25.}

\begin{abstract}
In this paper we study the skein algebras of marked surfaces and the skein modules of marked 3-manifolds. Muller showed that skein algebras of totally marked surfaces may be embedded in easy to study algebras known as quantum tori. We first extend Muller's result to permit marked surfaces with unmarked boundary components. The addition of unmarked components allows us to develop a surgery theory which enables us to extend the Chebyshev homomorphism of Bonahon and Wong between skein algebras of unmarked surfaces to a ``Chebyshev-Frobenius homomorphism'' between skein modules of marked 3-manifolds. We show that the image of the Chebyshev-Frobenius homomorphism is either transparent or skew-transparent. In addition, we make use of the Muller algebra method to calculate the center of the skein algebra of a marked surface when the quantum parameter  is not a root of unity.
\end{abstract}

\maketitle

\section{Introduction}

\subsection{Skein modules of 3-manifolds} In this paper we study the skein modules of marked 3-manifolds, which have connections to many important objects like character varieties, the Jones polynomial,  Teichm\"uller spaces, and cluster algebras. Skein modules serve as a bridge between classical and quantum topology, see e.g. \cite{Bullock,Ka,Le,Le5,Marche,Turaev}.

By a {\em marked 3-manifold} we mean a pair $(M,\cN)$, where $M$ is an oriented  3-manifold with (possibly empty) boundary $\partial M$ and a 1-dimensional oriented submanifold $\cN \subset \partial M$ such that each connected component of $\cN$ is diffeomorphic to the  open interval $(0,1)$. 
\def\Sx{\cS_\xi}
By an {\em $\cN$-tangle in $M$} we mean a compact 1-dimensional non-oriented submanifold $T$ of $M$ equipped with a normal vector field, called the {\em framing}, such that $\partial T= T \cap \cN$ and the framing at each boundary point of $T$ is a positive tangent vector of $\cN$. Two $\cN$-tangles are {\em $\cN$-isotopic} if they are isotopic in the class of $\cN$-tangles. 

For a non-zero complex number $\xi$ the skein module $\Sx\MN$ is the $\BC$-vector space freely spanned by $\cN$-isotopy classes of $\cN$-tangles modulo the local relations described in Figure \ref{fig:skein}. For a detailed explanation  see Section \ref{s.skeinmodules}.

\FIGc{skein}{Defining relations of skein module. From left to right: skein relation, trivial knot relation, and trivial arc relation. Here $q=\xi$.}{1.5cm}
When $\cN=\emptyset$ we don't need the third relation (the trivial arc relation), and in this case the skein module was introduced independently by J. Przytycki \cite{Przy} and V. Turaev \cite{Turaev,Turaev2}. The skein relation and the trivial knot relations were introduced by Kauffman \cite{Ka} in his study of the Jones polynomial.

The trivial arc relation was introduced by G. Muller in \cite{Muller} where he defined the skein module of marked surfaces, which was then generalized to marked 3-manifolds by the first author \cite{Le3}. It turns out that the extension to include the marked set $\cN$ on the boundary $\partial M$ makes the study of the skein modules much easier both in technical and conceptual perspectives.

\subsection{The Chebyshev-Frobenius homomorphism} \lbl{sec.Che}

When $\cN=\emptyset$,  Bonahon and Wong \cite{BW1} constructed a remarkable map between two skein modules, called the Chebyshev homomorphism, which plays an important role in the theory. One main result of this paper is to extend Bonahon and Wong's Chebyshev homomorphism to the case where $\cN\neq \emptyset$ and to give a conceptual explanation of its existence from basic facts about $q$-commutative algebra.

A 1-component $\cN$-tangle $\al$ is  diffeomorphic to either the circle $S^1$ or  the closed interval $[0,1]$; we call $\al$ an {\em $\cN$-knot} in the first case and {\em $\cN$-arc} in the second case.
For a 1-component $\cN$-tangle $\al$ and an integer $k\ge 0$, write $\al^{(k)} \in \cS_\xi(M)$ for the {\em $k$th framed power of $\al$} obtained by stacking $k$ copies of $\al$ in a small neighborhood of $\al$ along the direction of the framing of $\al$. Given a polynomial $P(z) = \sum c_i z^i \in \BZ[z]$, the {\em threading} of $\al$ by $P$ is given by $P^{\fr}(\al) = \sum c_i \al^{(i)} \in \cS_\xi(M)$.

The Chebyshev polynomials of type one $T_n(z) \in \BZ[z]$ are defined recursively as
\be 
T_0(z)=2, \ \ \ T_1(z)=z, \ \ \ T_n(z) = zT_{n-1}(z)-T_{n-2}(z), \ \ \forall n \geq 2. \lbl{eq.Che}
\ee

The extension of Bonahon and Wong's result to marked 3-manifolds is the following.

\begin{theorem} [see Theorem \ref{t.ChebyshevFrobenius}] \lbl{thm.1}

Suppose
 $(M,\cN)$ is a marked 3-manifold and $\xi$ is a complex root of unity. Let $N$ be the order of $\xi^4$, i.e. the smallest positive integer such that $\xi^{4N}=1$. Let $\ve =\xi^{N^2}$. Then there exists a unique $\BC$-linear map $\Phi_\xi: \cS_\ve(M,\cN) \to \cS_\xi(M,\cN)$ such that for any $\cN$-tangle $T = a_1 \cup \cdots \cup a_k \cup \al_1 \cup \cdots \cup \al_l$ where the $a_i$ are $\cN$-arcs and the $\al_i$ are $\cN$-knots,

\begin{align*}
\Phi_\xi(T) & = a_1^{(N)} \cup \cdots \cup a_k^{(N)} \cup (T_N)^{\fr}(\al_1) \cup \cdots \cup (T_N)^{\fr}(\al_l) \quad \text{in }\Sx\MN\\
&: = \sum_{0\le j_1, \dots, j_l\le N} c_{j_1} \dots c_{j_l}  a_1^{(N)}
 \cup \dots \cup  a_k^{(N)} \cup \, \al_1^{(j_1)} \cup \cdots \cup \al_l^{(j_l)} \quad \text{in }\Sx\MN,
\end{align*}
where $T_N(z) = \sum_ {j=0}^N c_j z^j$.
\end{theorem}

We call $\Phi_\xi$ the {\em Chebyshev-Frobenius homomorphism}. Our construction and proof are independent of the previous results of \cite{BW1} (which requires the quantum trace map \cite{BW0}) and \cite{Le2}. This is true even for the case $\cN=\emptyset$. 

Note that when $T$ has only arc components, then $\Phi_\xi(T)$ is much simpler as it can be defined using monomials and no Chebyshev polynomials are involved. The main strategy we employ is to understand the Chebyshev-Frobenius homomorphism for this simpler case, then show that the knot components case can be reduced to this simpler case.

\subsection{Skein algebras of marked surfaces} In proving Theorem \ref{thm.1} we prove several results  on the skein algebras of marked surfaces which are of independent interest. 
The first  is an extension of the a result of Muller \cite{Muller} from totally marked surfaces to marked surfaces.

By a {\em marked surface} we mean a pair $\SM$, where $\Sigma$ is an oriented surface with (possibly empty) boundary $\pS$ and a finite set $\cP\subset \pS$. Define $\Sx\SM= \Sx\MN$, where $M = \Sigma \times (-1,1)$ and $\cN=\cP \times (-1,1)$. Given two $\cN$-tangles $T, T'$ define the product $T T'$ by stacking $T$ above $T'$. This gives $\Sx\SM$ an algebra structure, which was first introduced by Turaev \cite{Turaev} for the case $\cP=\emptyset$ in connection with the quantization of the Atiyah-Bott-Weil-Petersson-Goldman symplectic structure of the character variety. The algebra $\Sx\SM$ is closely related to the quantum Teichm\"uller space and the quantum cluster algebra of the surface. If $\beta$ is an {\em unmarked boundary component of $\pS$}, i.e. $\beta \cap \cP=\emptyset$, then $\beta$ is a central element of $\Sx\SM$. Thus $\Sx\SM$ can be considered as an algebra over $\BC[\cH]$, the polynomial algebra generated by $\cH$ which is the set of all unmarked boundary components of $\pS$.

{\em A $\cP$-arc} is a path $a: [0,1]\to \Sigma$ such that $a(0), a(1) \in \cP$ and $a$ maps $(0,1)$ injectively into $\Sigma \setminus \cP$. A {\em quasitriangulation} of $\SM$   is a collection $\D$ of $\cP$-arcs which cut $\Sigma$ into triangles and holed monogons (see \cite{Penner} and Section \ref{sec:markedsurface}). Associated to a quasitriangulation $\D$ is a vertex matrix $P$, which is an anti-symmetric $\D \times \D$ matrix with integer entries. See Section \ref{sec:markedsurface} for details.
Define the Muller algebra 
\be 
 \sX_\xi(\D) = \BC[\cH] \la a^{\pm 1}, a \in \D \mid ab = \xi^{P_{a,b}}\,  ba\ra,
 \lbl{eq.pre1}
 \ee
which was introduced by Muller \cite{Muller} for the case when $\cH=\emptyset$. An algebra of this type is  called a {\em quantum torus}.  A quantum torus is like an algebra of Laurent polynomials in several variables which $q$-commute, i.e. $ab= q^k ba$ for a certain integer $k$. Such a quantum torus is Noetherian, an Ore domain, and has many other nice properties. In particular, $\sX_\xi(\D)$ has a ring of fractions $\tiX_\xi(\D)$ which is a division algebra. The $\BC[\cH]$-subalgebra of $\sX_\xi(\D)$ generated by $a\in \D$ is denoted by $\sX_{+,\xi}(\D)$.

\def\tvD{\tilde \varphi_\D}
Then we have the following result.
\begin{theorem} [See Theorem \ref{r.torus}]\lbl{thm.2}
 Assume $\D$ is a quasitriangulation of marked surface $\SM$. Then there is a natural algebra embedding 
 $$\varphi_\D: \Sx\SM \embed \sX_\xi(\D)$$ such that $\varphi_\D(a)= a$ for all $a\in \D$. The image of $\varphi_\D$ is sandwiched between $\sX_{+,\xi}(\D)$ and $\sX_{\xi}(\D)$. The algebra $\Sx\SM$ is an Ore domain, and $\varphi_\D$ induces an isomorphism 
 $ \tvD: \widetilde \Sx\SM \overset \cong \longrightarrow \tiX_\xi(\D)$ where $\widetilde \Sx\SM $ is the  division algebra of $ \Sx\SM$.
\end{theorem}

In the case when $\cH=\emptyset$, Theorem \ref{thm.2} was proved by Muller \cite{Muller}.
The significance of the theorem is that, as $\Sx\SM$ is sandwiched between $\sX_{+,\xi}(\D)$ and $\sX_\xi(\D)$,  many problems concerning  $\Sx\SM$ are reduced to problems concerning the quantum torus $\sX_\xi(\D)$ which is algebraically much simpler.

\subsection{Surgery} One important feature of the inclusion of unmarked boundary components is that we can develop a surgery theory.
One can consider the embedding $\vpD: \Sx\SM \embed \sX_\xi(\D)$ as a coordinate system of the skein algebra which depends on the quasitriangulation $\D$. A function  on $ \Sx\SM$ defined through the coordinates  makes sense only if it is independent of the coordinate systems. One such problem is discussed in the next subsection. To help with this independence problem we develop a surgery theory of coordinates in Section \ref{sec.surgery} which describes how the coordinates change under certain modifications of the marked surfaces. We will consider  two such modifications: one is to add a marked point and the other one is to plug a hole, (i.e. glue a disk to a boundary component with no marked points). Note that the second one changes the topology of the surface, and is one of the reasons why we want to extend Muller's result to allow unmarked boundary components. Besides helping with proving the existence of the Chebyshev-Frobenius homomorphism, we think our surgery theory will find more applications elsewhere.

\subsection{Independence of triangulation problem} Suppose $\SM$ is a {\em triangulable} marked surface, i.e.  every boundary component of $\Sigma$ has at least one marked point and $\SM$ has a quasitriangulation. In this case a quasitriangulation is  called  simply a {\em triangulation}. Let $\D$ be a  triangulation  of $\SM$. Let $\xi$ be a non-zero complex number (not necessarily a root of unity), $N$ be a positive integer, and $\ve=\xi^{N^2}$. 

From the presentation \eqref{eq.pre1} of the quantum tori $\sX_\ve (\D)$ and $\sX_\xi (\D)$, one sees that there is an algebra homomorphism, called the Frobenius homomorphism,
$$ F_N : \sX_\ve (\D) \to \sX_\xi(\D), \quad\text{given by}\ F_N(a)= a^N \quad \forall a \in \D,$$
which is injective, see Proposition \ref{r.Frobenius}.

Identify $\cS_\nu \SM$ with a subset of $\sX_\nu(\D)$ for $\nu=\xi$ and $\nu=\ve$ via the embedding $\vpD$. Consider the following diagram 
\be \lbl{dia.sx0}
\begin{tikzcd}
\cS_\ve \SM  \arrow[hookrightarrow]{r} \arrow[dashed,"?"]{d}& \sX_\ve(\D)  \arrow[d, "F_N"] \\
\cS_\xi \SM \arrow[hookrightarrow]{r} & \sX_\xi(\D)
\end{tikzcd}
\ee

We consider the following questions about $F_N$:

A. For what  $\xi \in \BC\setminus \{0\}$ and $N\in \BN$ does   $F_N$  restrict to a map from $\cS_{\ve}\SM$ to $\cS_{\xi}\SM$ and the restriction does not depend on the triangulation $\D$?

B.  If $F_N$ can restrict to such a map,  can one define the restriction of $F_N$ onto $\cS_{\ve}\SM$ in an intrinsic way, not referring to any triangulation $\D$?

It turns out that the answer to Question A is that $\xi$ must be a root of unity, and $N$ is the order of $\xi^4$. See Theorem \ref{thm.A}. Then Theorem \ref{thm.surface}, answering Question B, states that under these assumptions on $\xi$ and $N$, the restriction of $F_N$ onto $\cS_{\ve}\SM$ can be defined in an intrinsic way without referring to any triangulation. Explicitly, if $a$ is a $\cP$-arc, then $F_N(a)= a^N$, and if $\al$ is a $\cP$-knot, then $F_N(\al)= T_N(\al)$. From these results and  the functoriality of the skein modules we can prove Theorem \ref{thm.1}.

\subsection{Centrality and transparency}
With the help of Theorem \ref{thm.2} we are able to determine the center of $\Sx\SM$ for generic $\xi$.
\begin{theorem}[See Theorem \ref{thm.center1}] Suppose $\SM$ is a marked surface with at least one quasitriangulation and $\xi$ is not a root of unity. Then the center of $\Sx\SM$ is the $\BC$-subalgebra generated by $z_\beta$ for each connected component $\beta$ of $\pS$. Here if $\beta \cap \cP=\emptyset$ then $z_\beta= \beta$, and if $\beta \cap \cP\neq \emptyset$ then $z_\beta$ is the product of all $\cP$-arcs lying in $\beta$.
\end{theorem}

The center of an algebra is important, for example, in questions about the representations of the algebra. When $\xi$ is a root of unity and $\cP=\emptyset$, the center of $\Sx(\Sigma, \emptyset)$ is determined in \cite{FKL} and is instrumental in proving the main result there, the unicity conjecture of Bonahon and Wong. In a subsequent paper we will determine the center of $\Sx\SM$ for the case when $\xi$ is a root of unity. In \cite{PS2}, the center of a {\em reduced} skein algebra, which is a quotient of $\Sx\SM$, was determined for the case when $\xi$ is not a root of unity.

An important notion closely related to centrality in skein algebras of marked surfaces is {\em transparency} and {\em skew-transparency} in skein modules of marked 3-manifolds.  Informally an element $x \in \cS_\xi\MN$ is transparent if passing a strand of an $\cN$-tangle $T$ through $x$ does not change the the value of the union $x \cup T$.  

We generalize the result in \cite{Le2} for unmarked 3-manifolds to obtain the following theorem. 

\begin{theorem} (See Theorem \ref{r.mutransparent})
Suppose $\MN$ is a marked 3-manifold, $\xi$ is a root of unity, $N=\text{ord}(\xi^4)$, and $\ve=\xi^{N^2}$. Suppose $\xi^{2N} = 1$.  Let $\Phi_\xi: \cS_\ve\MN \to \cS_\xi\MN$ be  the Chebyshev-Frobenius homomorphism. Then the image of $\Phi_\xi$ is transparent in the sense that if $T_1,T_2$ are $\cN$-isotopic $\cN$-tangles disjoint from another $\cN$-tangle $T$, then in $\cS_\xi\MN$ we have
$$ \Phi_\xi(T) \cup T_1 = \Phi_\xi(T) \cup T_2.$$
\end{theorem}

Note that since $\text{ord}(\xi)=N$, we have either $\xi^{2N}=1$ or $\xi^{2N}=-1$. When $\xi^{2N}=-1$, the corresponding result is that the image of $\Phi_\xi$ is {\em skew-transparent}, see 
Theorem \ref{r.mutransparent}. Theorem \ref{r.mutransparent} can be depicted pictorially as the identity in Figure \ref{fig:transparent}.

\FIGc{transparent}{Applying the Chevyshev-Frobenius homomorphism to an $\cN$-tangle makes it transparent or skew-transparent.}{1.5cm}

\subsection{Chebyshev polynomial of $q$-commuting variable}

In the course of proving the main theorem, we apply the following simple but useful result given in Proposition \ref{p.prop31} relating Chebyshev polynomials and $q$-commuting variables.

\begin{proposition}[See Proposition \ref{p.prop31}]
Suppose $K,E$ are variables and $q$ an indeterminate such that $KE=q^2EK$ and $K$ is invertible. Then for any $n \ge 1$,
$$
T_n(K+K^{-1}+E) = K^n + K^{-n} + E^n + \sum_{r=1}^{n-1}\sum_{j=0}^{n-r} c(n,r,j)[E^rK^{n-2j-r}],
$$
where $c(n,r,j) \in \BZ[q^{\pm 1}]$ is given explicitly by  \eqref{eq.cnrj} as a ratio of $q$-quantum integers, and in fact $c(n,r,j) \in \BN[q^{\pm 1}]$. Besides, if $q^2$ is a root of unity of order $n$, then $c(n,r,j)=0$.  
\end{proposition}

In particular, if $q^2$ is a root of unity of order $n$, then
$$ 
T_n(K+K^{-1}+E) = K^n + K^{-n} + E^n.
$$
The above identity
 was first proven in \cite{BW1}, as an important case of the calculation of the Chebyshev homomorphism. See also \cite{Le2}. An algebraic proof of identities of this type is given in \cite{Bonahon}. Here the identity follows directly from  the Proposition \ref{p.prop31}. The new feature of Proposition \ref{p.prop31} is that it deals with generic  $q$ and may be useful in the study of the positivity of the skein algebra, see \cite{Thurston,Le:Positive}.

\subsection{Structure of the paper}
A brief summary of each section of the paper is as follows.

\begin{enumerate}
\setcounter{enumi}{+1}
\item \textbf{Quantum tori and Frobenius homomorphism.} This section defines the quantum torus abstractly. Also in this section is the definition of the Frobenius homomorphism between quantum tori.
\item \textbf{Chebyshev polynomials and quantum tori.} This section is a review of Chebyshev polynomials and some computations of Chebyshev polynomials with $q$-commuting variables.
\item \textbf{Skein modules of 3-manifolds.} This section defines the skein module of a marked 3-manifold $(M,\cN)$ and some related terminology.
\item \textbf{Marked surfaces.} Here we define some terminology related to marked surfaces $\SM$, including $\cP$-triangulations and $\cP$-quasitriangulations, as well as the vertex matrix which is used to construct the quantum torus into which the skein algebra embeds, known as a Muller algebra.
\item \textbf{Skein algebras of marked surfaces.} Here, Muller's skein algebra of totally marked surfaces is extended to the case where $(\Sigma,\cP)$ is not totally marked and basic facts are given, such as how to embed the skein algebra in the Muller algebra.
\item \textbf{Modifying marked surfaces.} We describe a surgery theory that describes what happens to a skein algebra of a marked surface when marked points are added or when a hole corresponding to an unmarked boundary component is plugged.
\item \textbf{Chebyshev-Frobenius homomorphism.} We prove the existence of the Chebyshev-Frobenius homomorphism between marked 3-manifolds and marked surfaces in this section.
\item \textbf{Image of $\Phi_\xi$ and (skew-)transparency} We show that the Chebyshev-Frobenius homomorphism is (skew-)transparent. 
\item \textbf{Center of the skein algebra for $q$ not a root of unity} We utilize the Muller algebra to give a short argument which finds the center of the skein algebra when $q$ is not a root of unity.
\end{enumerate}

\subsection{Acknowledgments} The authors would like to thank F. Bonahon, C. Frohman, J. Kania-Bartozynska, A. Kricker, G. Masbaum, G. Muller,  A. Sikora, D. Thurston, and the anonymous referee for helpful discussions. The first author  would like to thank the CIMI Excellence Laboratory, Toulouse, France, for inviting him on a Excellence Chair during the period of January -- July 2017 when part of this work was done. The first author is supported in part by NSF grant DMS 1811114.

\section{Quantum tori and Frobenius homomorphism}\lbl{s.qtorus}
In this section  we survey the basics of quantum tori, Ore domains, and present the Frobenius homomorphism of quantum tori.
Throughout the paper $\BN,\BZ,\BQ,\BC$ are respectively the set of all non-negative integers, the set of integers, the set of rational numbers,  and the set of complex numbers.
 All rings  have unit and are associative. 
 
 In this section $R$ is a commutative Noetherian domain containing a distinguished  invertible element $q^{1/2}$. The reader should have in mind the example $R=\BZ[q^{\pm 1/2}]$.

\subsection{Weyl normalization}\lbl{sec.weylnormalization} Suppose $\mathcal{A}$ is an $R$-algebra, not necessarily commutative. Two elements $x,y \in \mathcal{A}$ are said to be {\em q-commuting} if there is $\mathcal{C}(x,y) \in \BZ$ such that $xy=q^{\mathcal{C}(x,y)}yx$. Suppose $x_1,x_2,\ldots,x_n \in \mathcal{A}$ are pairwise $q$-commuting. Then the {\em Weyl normalization} of $\prod_i x_i$ is defined by
$$
[x_1x_2 \ldots x_n]:=q^{-\frac{1}{2}\sum_{i<j}\mathcal{C}(x_i,x_j)}x_1x_2 \ldots x_n.
$$
It is known that the normalized product does not depend on the order, i.e. if $(y_1,y_2,\ldots,y_n)$ is a permutation of $(x_1,x_2,\ldots,x_n)$, then $[y_1y_2 \ldots y_n]=[x_1 x_2 \ldots x_n]$.

\def\cA{\mathcal A}
\subsection{Quantum torus}For a finite set $I$  denote by $\Mat(I \times I, \BZ)$ the set of all $I\times I$ matrices with entries in $\BZ$, i.e.
 $A \in \Mat(I \times I, \BZ)$ is a function $A: I \times I \to \BZ$.
We write $A_{ij}$ for $A(i,j)$.

Let $A \in \Mat(I \times I, \BZ)$ be antisymmetric, i.e. $A_{ij}= - A_{ji}$. 
Define
 the {\em quantum torus over $R$ associated to $A$ with basis variables $x_i, i \in I$ } by
\begin{align*}
\bT(A;R):= R\la x_i^{\pm 1} , i\in I\ra /(x_i x_j = q^{A_{ij}} x_j x_i).
\end{align*}
When $R$ is fixed, we write $\bT(A)$ for $\bT(A;R)$.
Let  $\bT_+(A)\subset \bT(A)$ be the subalgebra generated by $x_i, i\in I$.

Let $\BZ^I$ be the set of all maps $\bk: I \to \BZ$. For $\bk \in \BZ^I$ define the {\em normalized monomial} $x^\bk$ using the Weyl normalization
$$
x^\bk = \left[ \prod_{i \in I} x_i^{\bk(i)} \right].
$$
The set $\{x^\bk \mid \bk \in \BZ^I\}$ is an $R$-basis of $\bT(A)$, i.e. we have the direct decomposition
\be 
\lbl{eq.grading}
\bT(A) = \bigoplus_{\bk \in \BZ^I} R \, x^\bk.
\ee
Similarly, $\bT_+(A;R)$ is  free over $R$ with basis $\{x^\bk \mid \bk \in \BN^I\}$.

Define an anti-symmetric $\BZ$-bilinear form on $\BZ^I$ by

$$
\langle \bk, \bn \rangle_A:= \sum_{i,j\in I} A_{ij}\, \bk(i)\bn(j)
$$
The following well-known fact follows easily from the definition:
 For $\bk,\bn \in \BZ^I$, one has
\be \lbl{e.normalizedtorus}
x^\bk x^\bn = q^{\frac{1}{2} \langle \bk, \bn \rangle_A} x^{\bk+\bn}  = q^{\langle \bk,\bn \rangle_A}x^{\bn}x^{\bk},
\ee

In particular, for $n \in \BZ$ and $\bk \in \BZ^I$, one has 
\be 
\lbl{eq.power}
(x^\bk)^n = x^{n\bk}.
\ee
The first identity of  \eqref{e.normalizedtorus} shows that the decomposition \eqref{eq.grading} is a $\BZ^I$-grading of the $R$-algebra $\bT(A)$.

\def\cB{\mathcal B}
\subsection{Two-sided Ore domains, weak generation} Both $\bT(A;R)$ and $\bT_+(A;R)$ are two-sided Noetherian domains, see \cite[Chapter 2]{GW}. As any two-sided Noetherian domain is a two-sided Ore domain (see \cite[Corollary 6.7]{GW}), both $\bT(A;R)$ and $\bT_+(A;R)$ are two-sided Ore domains. Let us review some facts in the theory of Ore localizations. 

A {\em regular element} of a ring $D $ is any element $x\in D $ such that $xy\neq 0$ and  $yx\neq0$ for all non-zero $ y\in D $. 
If every $x\in D \setminus \{0\}$ is regular, we call $D $ a {\em domain}.

For a multiplicative subset $X\subset D $ consisting of regular elements of $D $, a ring $E$ is called a {\em ring of fractions for $D $ with respect to $X$} if $D $ is a subring of $E$ such that (i) every $x\in X$ is invertible in $E$ and (ii) every $e\in E$ has presentation $ e= d x^{-1} = (x')^{-1}(d') $ for $d,d'\in D $ and $x,x'\in X$. Then $D $ has a ring of fractions  with respect to $X$ if and only if $X$ is a two-sided Ore set, and in this case the left Ore localization $X^{-1} D $ and the right Ore localization $D  X^{-1}$ are the same and are isomorphic to $E$, see \cite[Theorem 6.2]{GW}. If $D $ is a domain and $X=D \setminus \{0\}$ is a two-sided Ore set, then $D $ is called an {\em Ore domain}, and $X^{-1} D = D  X^{-1}$ is a division algebra, called the {\em division algebra of $D $}.

\begin{proposition} \lbl{r.sandwich0}  Suppose $X$ is a two-sided Ore set of a ring $D$ and $D \subset D'\subset DX^{-1}$, where $D'$ is a subring of $DX^{-1}$.
 
(a) The set  $X$ is  a two-sided Ore set of $D'$ and $D'X^{-1}= D X^{-1}$.

(b) If $D$ is an Ore domain then so is $D'$, and both have the same division algebra.

\end{proposition}
\begin{proof} (a) Since $DX^{-1}$ is also a ring of fractions for $D'$ with respect to $X$, one has that  $X$ is a two-sided Ore set of $D'$ and $D'X^{-1}= D X^{-1}$.

(b) Let $Y=D\setminus \{0\}$. Since $D\subset D' \subset DX^{-1} \subset DY^{-1}$, the result follows from (a).
\end{proof}

\begin{corollary} \lbl{r.sandwich}  Suppose  $\bT_+(A)\subset D\subset \bT(A)$, where $D$ is a subring of $\bT(A)$. Then $D$ is an Ore domain and the embedding $ D\embed \bT(A)$ induces an $R$-algebra isomorphism from the division algebra of $D$ to that of $\bT$.
\end{corollary}

A subset $S$ of an $R$-algebra $D$ is said to {\em weakly generate} $D$ if $D$ is generated as an $R$-algebra by $S$ and the inverses of all invertible elements in $S$.
For example, $D$ weakly generates its Ore localization $D X^{-1}$.  
Clearly an $R$-algebra homomorphism $D \to D'$ is totally determined by its values on a set weakly generating $D$.



\subsection{Reflection anti-involution} \lbl{sec.reflection} The following is easy to prove, see \cite{Le3}.

\def\hchi{\hat \chi}
\begin{proposition} \lbl{r.reflection}   Assume that there is a $\BZ$-algebra homomorphism $\chi: R \to R$ such that $\chi(q^{1/2})= q^{-1/2}$ and $\chi^2= \id$, the identity map.
Suppose $A \in \Mat(I \times I, \BZ)$ is antisymmetric. There exists a unique
 $\BZ$-linear isomorphism $\hchi: \bT(A) \to \bT(A)$ such that $\hchi(r x^\bk )= \chi(r) x^\bk$ for all $r\in R$ and $\bk\in \BZ^I$, which is an anti-homomorphism, ie $\hchi(ab) = \hchi(b) \hchi(a)$. In addition, $\hchi^2=\id$.
\end{proposition}
 We call $\hchi$ the {\em reflection
anti-involution}.

\subsection{Frobenius homomorphism}\lbl{sec.frobhom}
\begin{proposition} \lbl{r.Frobenius}
Suppose $A \in \Mat(I \times I, \BZ)$ is an antisymmetric matrix and $N$ is a positive integer.
There is a unique $R$-algebra homomorphism, called the Frobenius homomorphism,
$$
F_N: \bT(N^2A) \to \bT(A)$$
such that $F_N(x_i) = x_i^N$. Moreover,  $F_N$ is injective.
\end{proposition}
\begin{proof} As the $x_i$ weakly generate $\bT(N^2A)$, the uniqueness is clear. If we define $F_N$ on the generators $x_i$ by $F_N(x_i)= x_i^N$, it is easy to check that the defining relations are respected by  $F_N$. Hence $F_N$ gives a well-defined $R$-algebra map. 

By \eqref{eq.power} one has $F_N(x^\bk)= x^{N \bk}$. This shows $F_N$ maps the $R$-basis 
$\{ x^\bk \mid \bk \in \BZ^I\}$ of $\bT(N^2A)$ injectively onto a subset of an $R$-basis of $\bT(A)$. Hence $F_N$ is injective.
\end{proof}


\newcommand{\qbinom}[2]{\left[\begin{matrix}
#1 \\ #2
\end{matrix}  \right]}

\section{Chebyshev polynomial and quantum torus} The Chebyshev polynomials of type one $T_n(z)\in \BZ[z]$ are  defined  recursively by
\begin{align*}
  T_0& =2 , \ T_1(z)= z , \ T_n(z) = z T_{n-1}(z) - T_{n-2}(z) \quad \text{for } \ n \ge 2. \\
\end{align*}
It is easy to see that for any invertible element $K$ in a ring,
\begin{align} \label{eTn}
T_n(K + K^{-1})& = K^n + K^{-n}.
\end{align}
We want to generalize the above identity and calculate $T_n(K+ K^{-1} +E)$, when $E$ is a new variable which $q$-commutes with $K$.

 Suppose $q$ is an indeterminate. For $n\in \BZ$ and $ k \in \BN$,
define as usual the quantum integer and the quantum binomial coefficient, which are elements of $\BZ[q^{\pm 1}]$, by
$$ [n]_q:= \frac{q^n - q^{-n}}{q - q^{-1}}, \quad   \qbinom nk_q = \prod_{j=1}^k \frac{[n-j+1]_q}{[j]_q}.$$

\begin{proposition}\lbl{p.prop31}
Suppose $K$ and $E$ are variables such that $KE= q^2 EK$ and $K$ is invertible.
Then for any $n \ge 1$,
\be
\lbl{eq.38}
T_n(K+ K^{-1} +E) = K^n + K^{-n} + E^n + \sum_{r=1}^{n-1}\sum_{j=0}^{n-r}  c(n,r,j) \left[ E^r K^{n-2j-r}\right]
\ee
where $c(n,r,j) \in \BZ[q^{\pm 1}]$ and is given by
\be
\lbl{eq.cnrj}
 c(n,r,j)= \frac{[n]_q}{[r]_q} \qbinom{n-j-1}{r-1}_q \qbinom{r+j-1}{r-1}_q.
 \ee
\end{proposition}

Here $[E^aK^b]$ is the Weyl normalization of of $E^aK^b$, i.e.
$$ \left [ E^a K^b \right] := q^{ab} E^a K^b = q^{-ab} K^b E^a.$$

\begin{proof}
One can easily prove the proposition by induction on $n$.
\end{proof}
\begin{corollary}[\cite{BW1,Bonahon}]\lbl{c.32}
Suppose $q^2$ is a root of unity of order exactly  $n$, then
$$
T_n(K+ K^{-1} +E) = K^n + K^{-n} + E^n.
$$
\end{corollary}
\begin{proof}
When $q^2$ is a root of unity of order $n$, then $[n]_q=0$ but $[r]_q \neq 0$ for any $ 1\le r \le n-1$. Equation \eqref{eq.cnrj} shows that $c(n,r,j)=0$ for all $ 1\le r \le n-1$. From \eqref{eq.38} we get the corollary.
\end{proof}

\begin{remark} \lbl{rem.34}
 If $n,k\in \BN$, then $\qbinom nk_q\in \BN[q^{\pm 1}]$, the set of Laurent polynomials with non-negative integer coefficients. 
From \eqref{eq.cnrj} it follows that  $c(n,r,j) \in \BN[q^{\pm 1}]$.
\end{remark}

\def\Cx{{\BC^\times}}
\def\ord{\mathrm{ord}}
\def\Sx{\cS_\xi}
\def\cL{\mathcal L}
\def\cT{\mathcal T}
\def\hPhi{\hat \Phi}
\def\Lplu{\left(  \raisebox{-13pt}{\incl{1.2 cm}{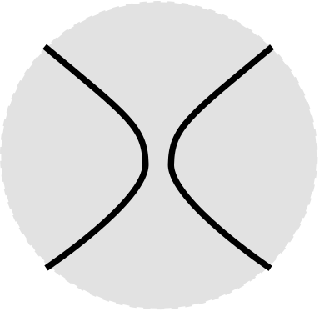}} \right) }
\def\Lpluss{\left(  \raisebox{-13pt}{\incl{1.2 cm}{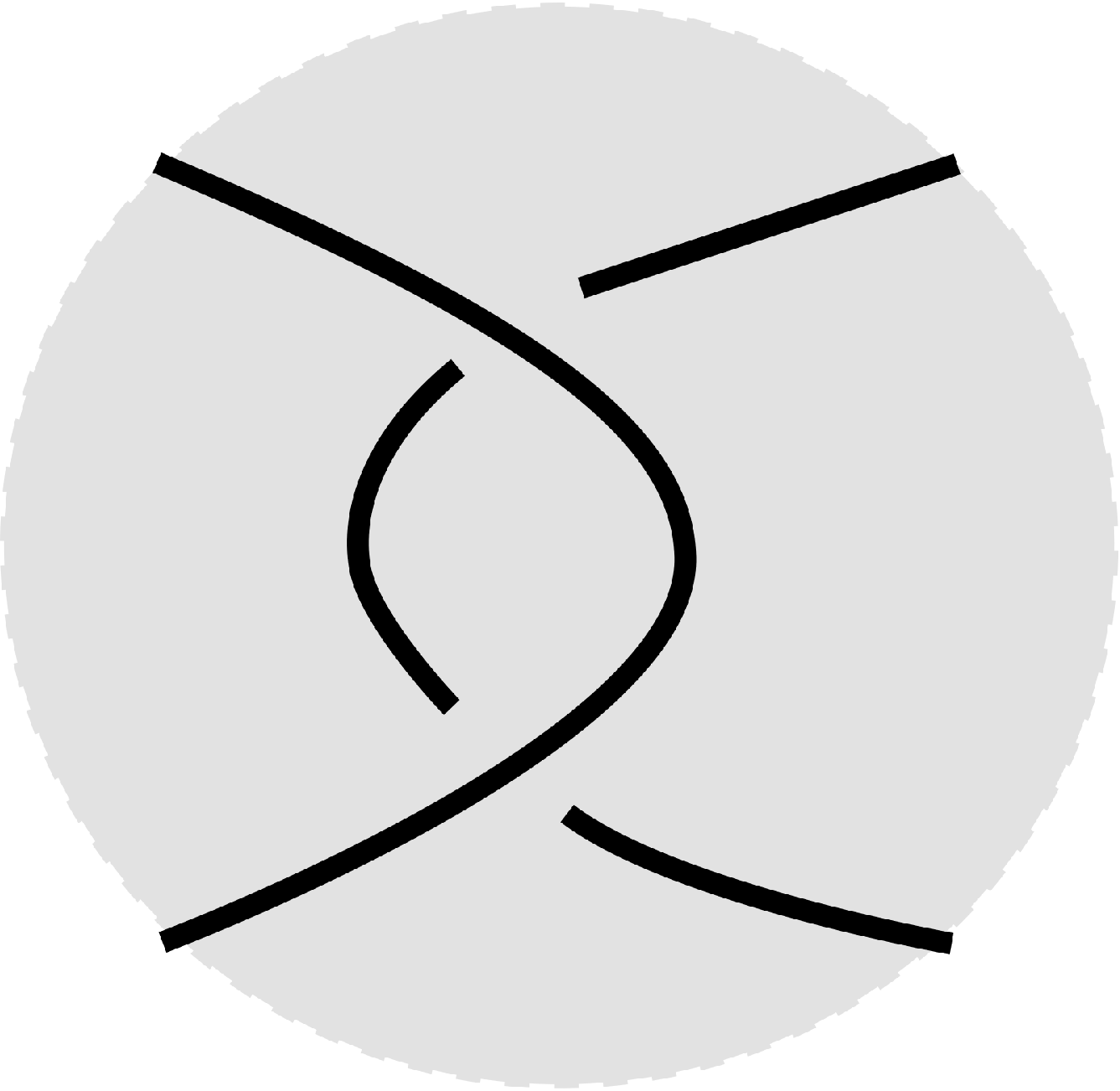}} \right) }
\def\uplus{\left(  \raisebox{-13pt}{\incl{1.2 cm}{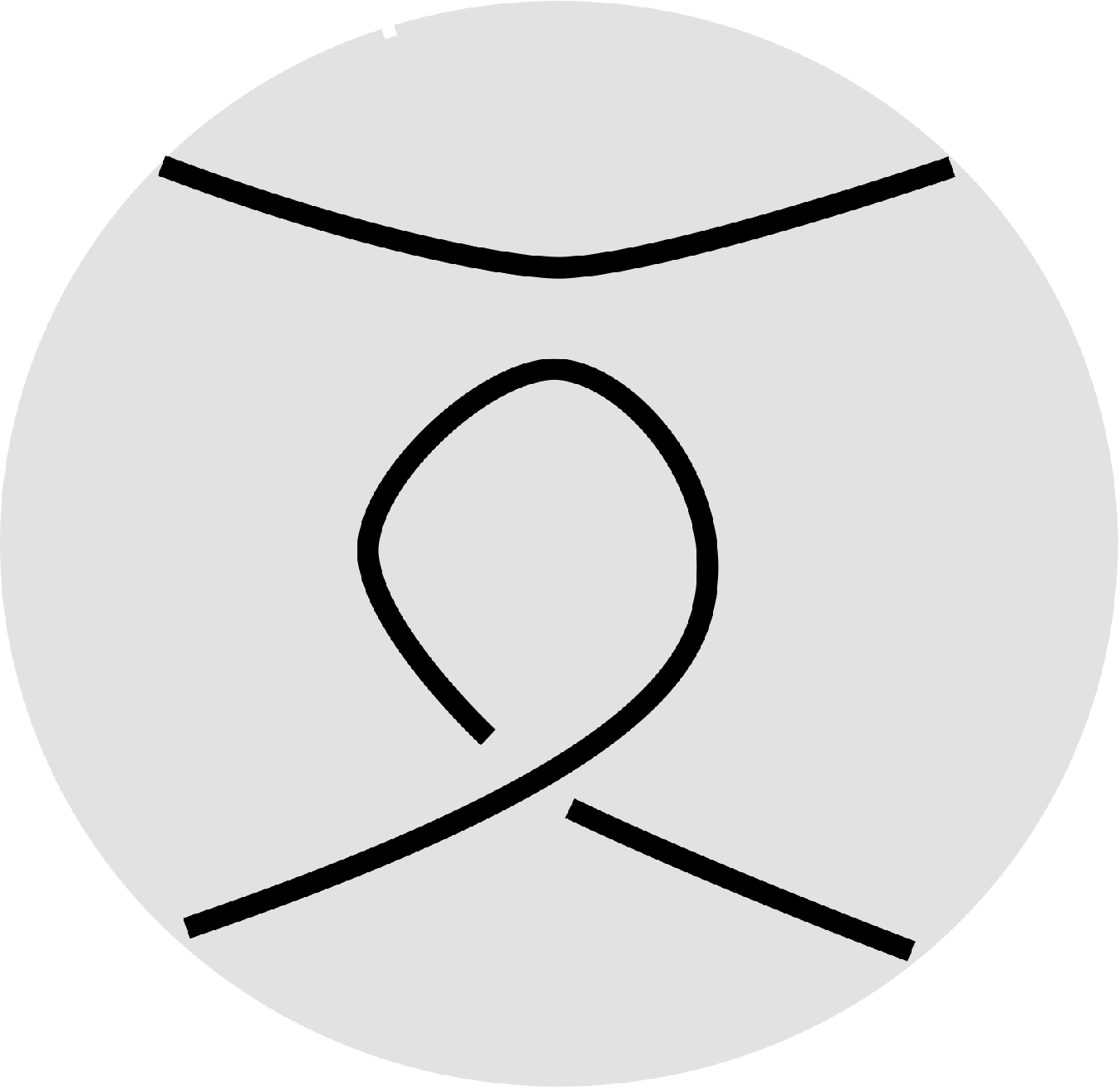}} \right) }
\def\uminus{\left(  \raisebox{-13pt}{\incl{1.2 cm}{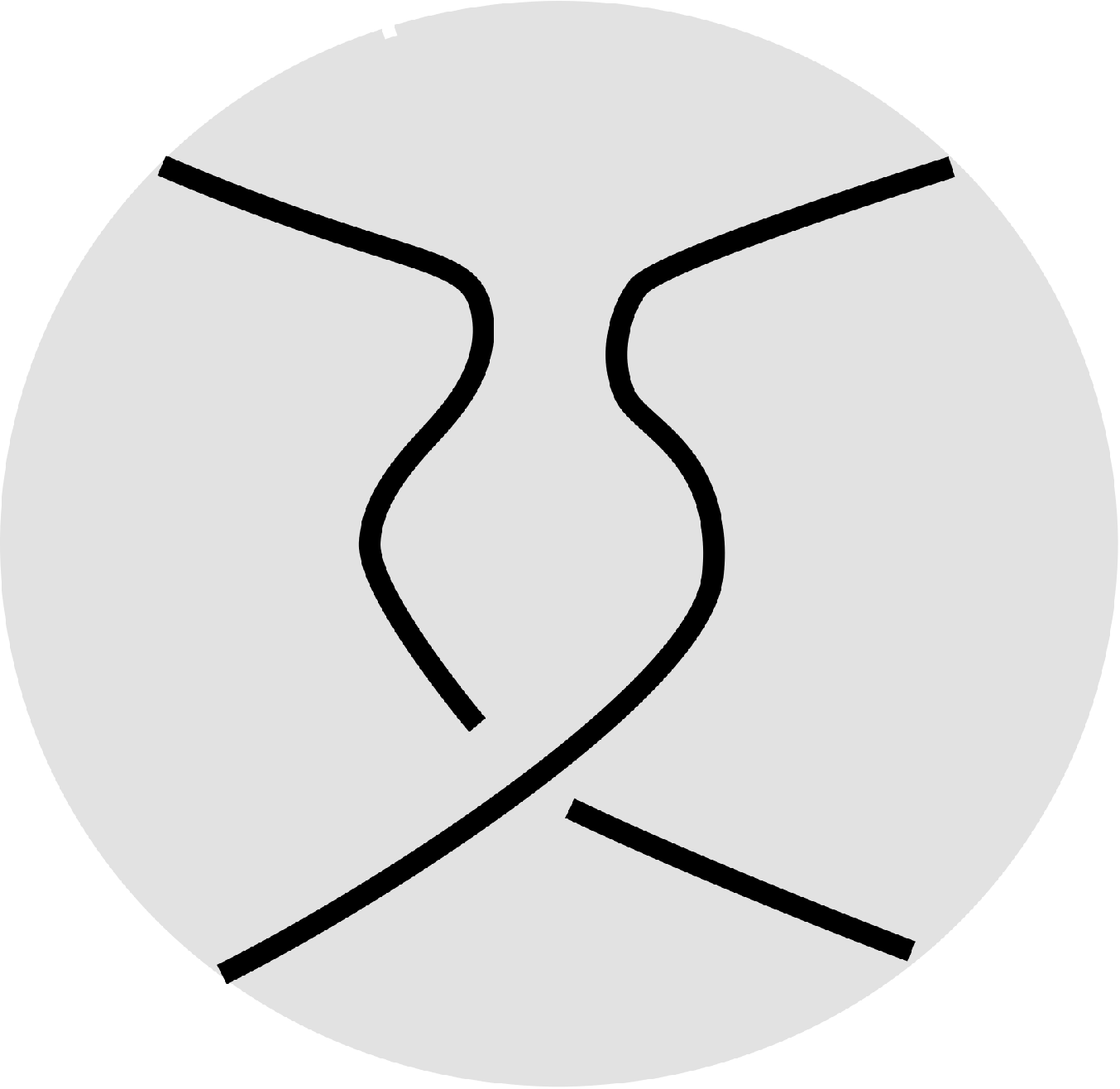}} \right) }
\def\LLL{\left(  \raisebox{-13pt}{\incl{1.2 cm}{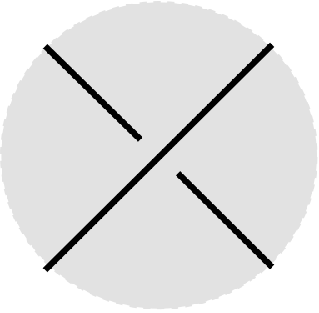}} \right) }
\def\Lminus{\left(  \raisebox{-13pt}{\incl{1.2 cm}{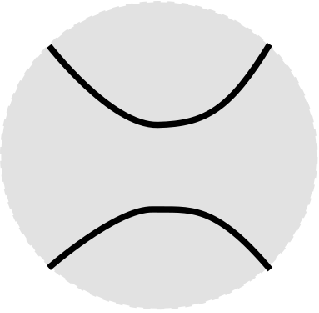}} \right) }
\def\negkinka{\raisebox{-13pt}{\incl{1.2 cm}{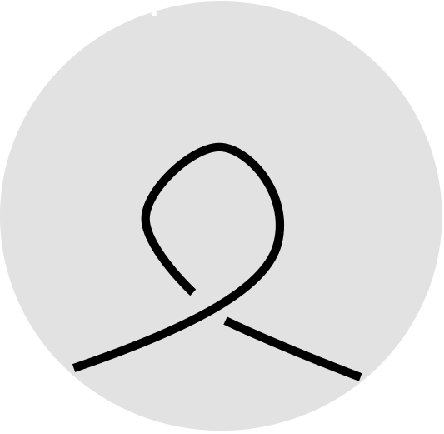}} }
\def\negkinkb{\raisebox{-13pt}{\incl{1.2 cm}{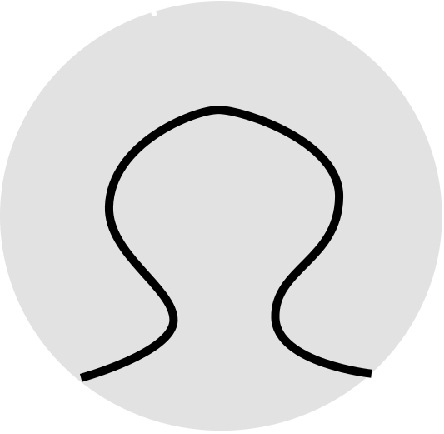}} }
\def\negkinkc{\raisebox{-13pt}{\incl{1.2 cm}{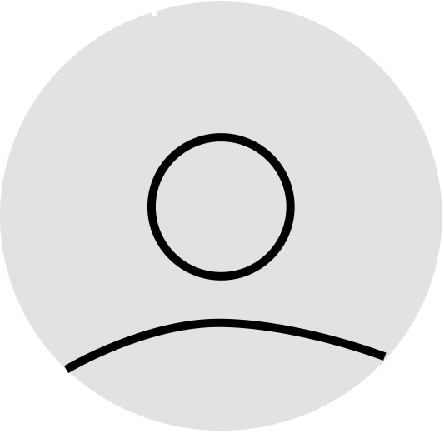}} }
\def\negkinkd{\raisebox{-13pt}{\incl{1.2 cm}{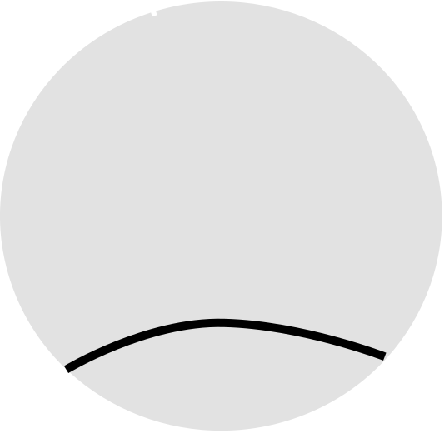}} }
\def\Col{\text{Col}}
 \section{Skein modules of 3-manifolds}\lbl{s.skeinmodules}
In this section we define the Kauffman bracket skein module of marked 3-manifolds, following   closely  \cite{Le3}. Recall that the ground ring $R$ is a commutative Noetherian domain, with a distinguished invertible element $q^{1/2}$ and  a $\BZ$-algebra involution $\chi:R\to R$ such that $\chi(q^{1/2})= q^{-1/2}$. In all figures throughout the paper, the framing of any tangle is assumed to be facing the reader unless otherwise specified.
\subsection{Marked 3-manifold} {\em A marked 3-manifold $\MN$} consists of
an oriented connected 3-manifold $M$  with (possibly empty) boundary $\pM$ and  a 1-dimensional oriented submanifold $\cN \subset \pM$ such that $\cN$ consists of a finite number of connected components, each of which is diffeomorphic to the interval $(0,1)$.

{\em An $\cN$-tangle $T$} in $M$ consists of a compact 1-dimensional non-oriented submanifold of $M$ equipped with a normal vector field
such that $T \cap
\cN = \partial T$ and at a boundary point $x\in \partial T \cap  \cN$, the normal vector is tangent to $\cN$ and agrees with the orientation of $\cN$. Here a normal vector field is a vector field which is not co-linear with the tangent space of $T$ at any point. This vector field is called the {\em framing} of $T$, and the vectors are called {\em framing vectors} of $T$.
 Two $\cN$-tangles are {\em $\cN$-isotopic} if 
 they are isotopic  through the class of  $\cN$-tangles. The empty set is also considered a  $\cN$-tangle which is $\cN$-isotopic only to itself.

\def\Se{\cS_\ve}
\def\Sx{\cS_\xi}
\def\tSe{\widetilde{\cS}_\ve}
\def\tSx{\widetilde{\cS}_\xi}
\def\Rel{\mathrm{Rel}}
 \subsection{Kauffman bracket skein modules} \lbl{s.def00}
Let $\cT\MN$ be the $R$-module freely spanned by the $\cN$-isotopy classes of $\cN$-tangles in $M$. The {\em Kauffman bracket skein module}  $\cS\MN$ is the quotient $\cS\MN= \cT\MN/\Rel_q$ where $\Rel_q$ is the $R$-submodule of $\cT\MN$ spanned by the skein relation elements, the trivial loop relation elements, and the trivial arc relation elements, where
 \begin{itemize}
 \item  $R$ is identified with the  $R$-submodule of $\cT\MN$ spanned by the empty tangle, via $c \to c\cdot \emptyset$,
 \item  a {\em skein relation element} is any element of the form  $T  -q T_+ - q^{-1} T_-$, where $T ,T_+,T_-$  are $\cN$-tangles  identical everywhere except in a ball in which they look like in Figure \ref{fig:skein1},

 \FIGc{skein1}{From left to right:  $T, T_+, T_-$.}{1.2cm}
\item a {\em trivial loop relation element} is any element of the form $\beta + q^2 + q^{-2}$, where $\beta$ is a trivial knot, i.e. a loop bounding a disk in $M$ with framing perpendicular to the disk.

\item   a {\em trivial arc relation element} is any $\cN$-tangle $T$ containing an $\cN$-arc $a$ for which there exists an arc $b \subset \cN$ such that $a \cup b$ bounds a continuous disk $D$ in $M$ such that $D \cap (T \setminus a) = \emptyset$. This situation is depicted in Figure \ref{fig:trivialarc}.

\FIGc{trivialarc}{The gray shaded area is a region of $\partial M$, the striped region is a continuous disk with boundary $a \cup b$, and $b \subset \cN$.}{1.5cm}
\end{itemize}

These relation elements are also depicted in Figure \ref{fig:skein} in the Introduction.

By \cite[Proposition 3.1]{Le3}, one also has the
 {\em  reordering relation} depicted in Figure \ref{fig:boundary} in $\cS\MN$.
\FIGc{boundary}{ Reordering relation: Here   $\cN$ is perpendicular to the page and  its perpendicular projection onto the page (and in the shaded disk) is the bullet denoted by $N$. The vector of orientation of $\cN$ is pointing to the reader.
There are two strands of the tangle coming to $\cN$ near $N$, with the lower one being depicted by the broken line}{2cm}

\begin{remark} Muller \cite{Muller} introduced Kauffman bracket skein modules for marked surfaces. Here we use a generalization of Muller's construction to marked 3-manifolds, introduced in \cite{Le3}. \end{remark}

\def\MNp{(M',\cN')}
\subsection{Functoriality}\lbl{sec.func}
 By a {\em morphism} $f: \MN \to \MNp$ between marked 3-manifolds 
we mean an orientation-preserving embedding $f: M \embed M'$ such that $f$ restricts to an orientation preserving embedding on $\cN$. Such a morphism induces  an $R$-module homomorphism $f_*: \cS\MNp \to \cS\MNp$ by $f_*(T)=T$ for any $\cN$-tangle $T$.

Given marked 3-manifolds $(M_i, \cN_i)$, $i=1,\dots, k$, such that $M_i\subset M, \cN_i \subset \cN$, and the $M_i$ are pairwise disjoint, then there is a unique $R$-linear map, called the {\em union map}
$$ \text{Union}:\prod_{i=1}^k \cS(M_i, \cN_i) \to \cS(M,\cN),$$
such that if $T_i$ is a $\cN_i$-tangle in $M_i$ for each $i$, then
$$ \text{Union} (T_1,\dots, T_k) = T_1 \cup \dots \cup T_k.$$
For $x_i\in \cS(M_i,\cN_i)$ we denote $\text{Union} (x_1,\dots, x_k)$ also by $x_1 \cup \dots \cup x_k$.

\def\ocP{\mathring {\cP}}

\def\PDDp{\Psi_{\D,\D'}}
\def\PSSp{\PSi_{\Sigma,\Sigma'}}

\section{Marked Surfaces}
\lbl{sec:markedsurface}


Here we present basic facts about marked surfaces and their quasitriangulations. Our marked surface is the same as a marked surface in \cite{Muller}, or a  ciliated surface in \cite{FG}, or a bordered surface with punctures in \cite{FST} if one consider a boundary component without amy marked points on it as a puncture.

\def\pr{\operatorname{pr}}

\subsection{Marked surface} \lbl{sec.41}

{\em A marked surface $\SM$} consists of
 a compact, oriented, connected surface $\Sigma$ with possibly empty boundary $\partial \Sigma$ and a finite set $\cP \subset  \pS$.  Points in $\cP$ are called {\em marked points}.
 A connected component of $\pS$ is {\em marked} if it has at least one marked point, otherwise it is {\em unmarked}. The set of all unmarked components is denoted by $\Hu$.  We call $\SM$ {\em totally marked} if $\Hu=\emptyset$, i.e. every boundary component has at least one marked point.

A {\em $\cP$-tangle} is an immersion $T: C\to \Sigma$, where $C$ is compact 1-dimensional non-oriented manifold, such that
\begin{itemize}
\item the restriction of $T$ onto the interior of $C$ is an embedding into  $\Sigma \setminus \cP$, and
\item $T$ maps the boundary of $C$ into $\cP$.
    \end{itemize}
    The image of a connected component of $C$ is called a {\em component} of $T$. When $C$ is homeomorphic to $S^1$, we call $T$ a {\em $\cP$-knot}, and when $C$ is $[0,1]$, we call $T$ a {\em $\cP$-arc}. 
Two $\cP$-tangles are {\em $\cP$-isotopic} if they are isotopic through the class of $\cP$-tangles.

\begin{remark}
We emphasize here that, unlike $\cN$-tangles in a marked 3-manifold $(M,\cN)$, a $\cP$-tangle $T: C \to \Sigma$ cannot actually be ``tangled'' since the restriction of $T$ to the interior of $C$ is an embedding into $\Sigma \setminus \cP$. We justify this terminology since we define the skein algebra of a surface in terms of $\cP \times (-1,1)$-tangles in $(\Sigma \times (-1,1), \cP \times (-1,1))$ and use $\cP$-tangles primarily as a tool to define and assist in working with $\cP \times (-1,1)$-tangles, such as the preferred $R$-basis $B_{\SM}$ of the skein algebra, see Subsection \ref{sec.sksurface}.
\end{remark}

A $\cP$-arc $x$ is called a {\em boundary arc}, or $x$ is {\em boundary}, if it is $\cP$-isotopic to a $\cP$-arc contained in $\pS$.
 A $\cP$-arc $x$ is called an {\em inner arc}, or $x$ is {\em inner}, if it is not a boundary arc.
 
 A $\cP$-tangle $T \subset \Sigma$ is {\em essential} if it does not have a component bounding a disk in $\Sigma$; such a component is either a smooth trivial knot in $\Sigma \setminus \cP$, or a $\cP$-arc bounding a disk in $\Sigma$. By convention, the empty set is considered an essential $\cP$-tangle.

\def\thD{\theta_\D}

\def\aa{d}
\subsection{Quasitriangulations} In most cases, a marked surface can be obtained by gluing together a collection of triangles and holed monogons along edges. Such a decomposition is called a quasitriangulation. We now give a formal definition. For details see \cite{Penner}.

A marked surface $\SM$ is said to be {\em quasitriangulable} if 
\begin{itemize}
\item there is at least one marked point, and
\item  $(\Sigma, \cP)$ is not a disk with $\le 2$ marked points, or an annulus with one marked point.
\end{itemize}

A {\em quasitriangulation} $\D$  of a quasitriangulable marked surface $\SM$, also called a {\em $\cP$-quasitriangulation} of $\Sigma$, is a collection of $\cP$-arcs such that
\begin{itemize}
\item[(i)] no two $\cP$-arcs in $\D$ intersect in $\Sigma \setminus \cP$ and no two are $\cP$-isotopic, and
\item[(ii)] $\D$ is maximal amongst all collections of $\cP$-arcs with the above property.
\end{itemize}

\def\bd{{\mathrm{bd}}}
\def\mon{{\mathrm{mon}}}
\def\sinn{{\mathrm{s.in}}}

 An element of $\D$ is also called an {\em  edge} of the $\cP$-quasitriangulation $\D$. Let $\D_\bd$ be the set of all {\em boundary edges}, i.e. edges which are boundary $\cP$-arcs. The complement $\D_\inn:= \D\setminus \D_\bd$ is the set of all {\em inner edges}. 
Then   $\D_\inn$ cuts $\Sigma$ into triangles  and {\em holed monogons} (see \cite{Penner} for exactly what is meant by this). Here a holed monogon is a region in $\Sigma$ bounded by an unmarked component of $\pS$ and a $\cP$-arc, see Figure \ref{fig:hole}.

\FIGc{hole}{Monogon bounded by $\cP$-arc $a_\beta$. The inner loop is a unmarked component $\beta$ of $\pS$, i.e. $\beta \in \cH$.}{1.3cm}

For an unmarked component $\beta\in \Hu$ let $a_\beta\in \D$ be the only edge on the boundary of the monogon containing $\beta$. We call $a_\beta$ the {\em monogon edge} corresponding to $\beta$, see Figure \ref{fig:hole}.  Denote by $\D_\mon \subset \D$ the set of all monogon edges.

The situation simplifies if $\SM$ is totally marked and quasitriangulable, i.e. $\Hu = \emptyset$, and $\SM$ is not a disk with $\le 2$ marked points. Then we don't have any monogon, and
instead of ``quasitriangulable" and ``quasitriangulation" we use the terminology ``triangulable" and ``triangulation". Thus every triangulable surface is totally marked.

\def\Rp{R[\Hu]}

\subsection{Vertex matrix} \lbl{sec.vmatrix} Suppose $a$ and $b$ are $\cP$-arcs which do not intersect in $\Sigma \setminus \cP$.  We define a number $P(a,b)\in \BZ$ as follows.  Removing an interior point of $a$ from $a$, we get two {\em half-edges} of $a$, each of which is incident to exactly one vertex in $\cP$. Similarly, removing an interior point of $b$ from $b$, we get two {\em half-edges} of $b$. Suppose $a'$ is a half-edge of $a$ and $b'$ is a half-edge of $b$, and $p\in \cP$. If one of $a', b'$  is not incident to $p$, set $P_p(a',b')=0$. If both $a', b'$ are incident to $p$, define
$P_p(a',b')$ as in Figure \ref{fig:vmatrix2}, i.e.
 $$P_p(a',b')= \begin{cases}
1  &\text{if $a'$ is clockwise to $b'$ (at vertex $p$)}\\
-1  &\text{if $a' $ is counter-clockwise to $b'$ (at vertex $p$)}.
\end{cases}
$$
\FIGc{vmatrix2}{$P_p(a',b')=1$ for the left case, and $P_p(a',b')=-1$ for the right one. Here the shaded area is part of $\Sigma$, and the arrow edge is part of a boundary edge. There might be other half-edges incident to $p$, and they maybe inside and outside the angle between $a'$ and $b'$. }{2.5cm}

Now define 
$$ P(a,b)= \sum P_p(a',b'),$$
where the sum is over all $p\in \cP$, all half-edges $a'$ of $a$, and all half-edges $b'$ of $b$.

Suppose $\D$ is a quasitriangulation of a  quasitriangulable marked surface $\SM$. Two distinct $a, b \in \D$ do not intersect in $\Sigma \setminus \cP$, hence we can define $P(a,b)$.
Let $P_\D \in \Mat(\D \times \D, \BZ)$, called the {\em vertex matrix} of $\D$, be the anti-symmetric $\D\times \D$ matrix defined by $P_\D(a,b)= P(a,b)$, with 0 on the diagonal.

\def\thD{\theta_\D}
\def\aa{d}
\def\Rp{R_\partial}

\begin{remark}
 The vertex matrix was  introduced in \cite{Muller}, where it is called the orientation matrix.
\end{remark}

\subsection{Intersection index}\lbl{sec.intersectionindex}

Given two 
 $\cP$-tangles $S,T$, the {\em intersection index} $\mu(S,T)$ is the minimal number of crossings between $S'$ and $T'$, over all transverse pairs $(S',T')$ such that $S'$ is  $\cP$-isotopic to $S$ and $T'$ is $\cP$-isotopic to $T$. Intersections at marked points are not counted. We say that $S$ and $T$ are {\em taut} if the number of intersection points of $S$ and $T$ in $\Sigma \setminus \cP$ is equal to $\mu(S,T)$. 

\begin{lemma}[\cite{FHS}]\lbl{l.FHS} Let $x_1,x_2,\ldots,x_n$ be a finite collection of essential $\cP$-tangles, then there are essential $\cP$-tangles $x_1',x_2',\ldots,x_n'$ such that,

\begin{itemize}
\item for all $i$, $x_i'$ is $\cP$-isotopic to $x_i$, and
\item for all $i$ and $j$, $x_i'$ and $x_j'$ are taut.
\end{itemize}
\end{lemma}

\section{Skein algebras of marked surfaces}

Throughout this section we fix a quasitriangulable 
marked surface $\SM$. The main result of this section is Theorem \ref{r.torus} which shows that the skein algebra of a quasitriangulable marked surface can be embedded into a quantum torus. We also discuss the flip operation on quasitriangulations.

\subsection{Skein module of a marked surface}\lbl{sec.sksurface} Let $M$ be the cylinder over $\Sigma$ and $\cN$ the cylinder over $\cP$, i.e. $M= \Sigma \times (-1,1)$ and $\cN= \cP \times (-1,1)$.
 We consider $\MN$ as a marked 3-manifold, where the orientation on each component of $\cN$  is given by the natural orientation of $I$. We will consider $\Sigma$ as a subset of $M$ by identifying $\Sigma$ with $\Sigma \times \{0\}$. There is a vertical projection $\pr:M \to \Sigma$, mapping $(x,t)$ to $x$.

Define $\cS\SM:= \cS\MN$. Since we fix $\SM$, we will use the notation $\cS=\cS\SM$ for the remainder of this section unless clarification is needed.

 An $\cN$-tangle $T$ in $M$ is said to have {\em vertical framing} if the framing vector at every point $p\in T$ is vertical, i.e. it is tangent to $p \times (-1,1)$ and has direction agreeing with the positive orientation of $(-1,1)$. 
 
 Suppose $T \subset  \Sigma$ is a $\cP$-tangle. Technically $T$ may not be an $\cN$-tangle in $M$  since several strands of $T$ may meet at the same point in $\cP$, which is forbidden in the definition of an $\cN$-tangle. We modify $T$ in  a small neighborhood of each point $p\in \cP$ by vertically moving the strands of $T$ in that neighborhood,  to get an $\cN$-tangle $T'$ in $M=\Sigma \times (-1,1)$ as follows. Equip $T$ with the vertical framing. If at a marked point $p$ there are $k=k_p$ strands
 $a_1, a_2, \dots, a_k$ of $T$  (in a small neighborhood of $p$) incident to $p$ and ordered in clockwise order, then we $\cN$-isotope these strands vertically so that $a_1$ is above $a_2$, $a_2$ is above $a_3$, and so on, see Figure \ref{fig:simul}. The resulting $T'$ is an $\cN$-tangle whose $\cN$-isotopy class only depends on the $\cP$-isotopy class of $T$. Define $T$ as an element in $\cS$ by
  \be
  \lbl{eq.simul}
   T: = q^{\frac 14 \sum_{p\in \cP} k_p(k_p-1)} T' \in \cS.
   \ee
 
  \FIGc{simul}{Left: There are 3 strands $a_1, a_2, a_3$ of $T$ coming to $p$, ordered clockwise. Right: The corresponding strands $a'_1, a'_2, a'_3$ of $T'$, with $a'_1$ above $a'_2$, and $a'_2$ above $a'_3$. Arrowed edges are part of the boundary, not part of the $\cP$-tanlges}{2cm}

The factor which is a power of $q$ on the right hand side is introduced so that $T$ is invariant under the reflection involution, see Subsection \ref{sec.involution} below.

\def\BB{B} 
\def\supp{\mathrm{supp}}

The set  $\BB_{\SM}$  of all $\cP$-isotopy classes of essential $\cP$-tangles in $\SM$ is a basis of the free $R$-module $\cS$, see
\cite[Lemma 4.1]{Muller}, and we will call $\BB_{\SM}$ the {\em preferred basis} of $\cS$.
For $0\neq   x\in \cS$ one has the finite presentation 
$$ x= \sum _{i \in I} c_i x_i,  \quad   c_i \in R\setminus \{0\}, \ x_i \in \BB_{\SM},$$
and we define the {\em support} of $x$ to be the set $\supp(x)= \{ x_i \mid i\in I\}$. For $z\in B_{\SM}$ define
\be 
 \mu(z,x)=   \max_{x_i \in \supp (x) } \mu (z, x_i).
 \ee
Here $\mu(z,x_i)$ is the geometric intersection index defined in Subsection \ref{sec.intersectionindex}.
 
\begin{remark}
Equation \eqref{eq.simul} describes the isomorphism between Muller's definition of the skein algebra of a totally marked surface $\SM$ in terms of multicurves of knots and arcs in $\SM$ and our definition of a skein algebra of a marked surface $\SM$ in terms of $\cP \times (-1,1)$-tangles in $(\Sigma \times (-1,1), \cP \times (-1,1))$.
\end{remark}

\subsection{Algebra structure and reflection anti-involution}\lbl{sec.involution}
For $\cN$-tangles $T_1, T_2$ in $(M,\cN)=(\Sigma \times (-1,1), \cP \times (-1,1))$ define the product
 $T_1 T_2$ as the result of stacking $T_1$ atop $T_2$ using the cylinder structure of $\MN$. More precisely, this means the following. Let $\iota_1: M \embed M$ be the embedding $\iota_1(x,t)= (x, \frac{t+1}2)$ and $\iota_2: M \embed M$ be the embedding
 $\iota_2(x,t)= (x, \frac{t-1}2)$. Then $T_1 T_2:= \iota_1(T_1) \cup \iota_2(T_2)$.
This product makes $\cS$ an $R$-algebra, which is non-commutative in general.

\def\refl{\mathrm{refl}}

Let $\hchi: \cS \to \cS$ be the bar homomorphism of \cite{Muller}, i.e. the $\BZ$-algebra anti-homomorphism defined by (i) $\hchi(x) = \chi(x)$ if  $x\in R$, and (ii) if $T$ is an $\cN$-tangle with framing $v$ then $\hchi(T)$ is $\refl(T)$ with the framing $-\refl(v)$, where $\refl$ is  the reflection which maps $(x,t)\to (x, -t)$ in $\Sigma \times (-1,1)$. It is clear that $\hchi$ is an anti-involution. An element $z\in \cS$ is {\em reflection invariant} if $\hchi(z)=z$.

The prefactor on the right hand side of \eqref{eq.simul} was introduced so that every $\cP$-tangle $T$  is reflection invariant as an element of $\cS$. The preferred basis $\BB_{\SM}$ consists of reflection invariant elements.

Suppose  $T$ is a $\cP$-tangle with components $x_1,\dots, x_k$. By the reordering relation (see Figure \ref{fig:boundary}), any two components $x_i, x_j$ are $q$-commuting as elements of $\cS$ (and note that $x_ix_j=x_jx_i$ if at least one is a $\cP$-knot), and
$$ T= [x_1 x_2 \dots x_k]  \quad \text{in }\ \cS,$$
where on the right hand side we use the Weyl normalization, see Subsection \ref{sec.weylnormalization}.

\subsection{Functoriality} \lbl{sec.surfunc}

Let $(\Sigma',\cP')$ be a marked surface such that $\Sigma'\subset \Sigma$ and $\cP'\subset \cP$. The morphism $\iota: (\Sigma',\cP') \embed \SM$ given by the natural embedding induces an $R$-algebra homomorphism $\iota_*:\cS(\Sigma', \cP')\to \cS\SM$.

\begin{proposition} \lbl{r.func} Suppose $\cP'\subset \cP$. Then   $\iota_*:\cS(\Sigma, \cP')\to \cS\SM$ is injective.
\end{proposition}
\begin{proof}This is because the preferred basis $B_{(\Sigma,\cP')}$ is a subset of $B_{\SM}$.
\end{proof}


\def\fM{\mathfrak M}
\def\cB{\Hu}

\def\tvpD{\tilde {\vp}_\D}
\def\ttS{\widetilde{\cS}}
\def\inqx{\text{in}_{q,x}}
\def\vD{\varphi_\D}
\def\vDp{\tau}
\def\vDpt{\tilde{\tau}}
\def\fm{\mathfrak m}
\def\Rb{{R[\Hu]}}
\def\bfM{\widetilde {\fM}}
\def\fn{\mathfrak n}

\subsection{Quantum torus associated to the vertex matrix for a marked surface}\lbl{sec.generalqtorus} Recall that an unmarked component is 
 a connected component of $\pS$ not containing any marked points, and $\Hu$ is the set of unmarked components. 
It is clear that every $\beta\in \Hu$ is  in the center of $\cS$. Note that two distinct elements of $\Hu$ are not $\cP$-isotopic since otherwise $\Sigma$ is an annulus with $\cP=\emptyset$, which is ruled out since $\SM$ is quasitriangulable. For $\bk \in \BN^\Hu$, define the following element of $\cS$: 
$$ \Hu^\bk:= \prod_{\beta \in \Hu} \beta^{\bk(\beta)} \in \cS.$$
The set $\{ \Hu^\bk \mid \bk \in \BN^\Hu\}$ is a subset of the preferred basis $\BB_{\SM}$. It follows  that the polynomial ring $R[\Hu]$ in the variables $\beta \in \Hu$ with coefficients in $R$ embeds as an $R$-subalgebra of $\cS$.  We will identify $R[\Hu]$ with this subalgebra of  $\cS$. Then $R[\Hu]$ is a subalgebra of the center of $\cS$, and hence we can consider $\cS$ as an $R[\Hu]$-algebra.

Let $\D$ be a quasitriangulation of $\SM$. By definition, each $a\in \Delta$ is a $\cP$-arc, and can be considered  as an element of the skein algebra $\cS$. From the reordering relation  we see that for each pair of $\cP$-arcs $a,b\in \D$,
\be
\lbl{eq.35}
ab = q^{P(a,b)} ba,
\ee
where $P\in \Mat(\D\times \D,\BZ)$ is the vertex matrix (see Subsection \ref{sec.vmatrix}).

 Let $\sX(\D)$ be the quantum torus over $R[\Hu]$ associated to $P$ with basis variables $X_a$, $a \in \D$.
That is,
\begin{align*}
\sX(\D) & = R[\Hu]\langle X_a^{\pm 1}, a \in \D\ra
/(X_aX_b=q^{P(a,b)}X_bX_a)
\end{align*}

As a free $R[\Hu]$-module, $\XD$ has a basis given by $\{X^\bn \mid \bn \in \BZ^\D\}$. As a free $R$-module,
$\XD$ has a basis given by $\{\Hu^ \bk \, X^\bn \mid \bk\in \BN^\Hu, \bn \in \BZ^\D\}$.

Let $\sX_+(\D)$ be the $\Rb$-subalgebra of $\XD$ generated by $X_a, a\in \D$. Then $\sX_+(\D)$ is a free $\Rb$-module with a basis given by $\{ X^\bn \mid \bn \in \BN^\D\}$ and a free $R$-module with preferred basis $B_{\D,+}:=\{\Hu^ \bk \, X^\bn \mid \bk\in \BN^\Hu, \bn \in \BN^\D\}$. Furthermore, $\sX_+(\D)$ has the following presentation as an algebra over $\Rb$: 
  $$ \sX_+(\D)= \Rb\la X_a , a\in \D\ra /( X_a X_b = q^{P(a,b)}X_b X_a).$$

The involution $\chi:R \to R$ extends to an involution $\chi : R[\Hu] \to R[\Hu]$ by $\chi(rx)= \chi(r) x$ for all $r\in R$ and $x = \Hu^\bk$ for all $\bk \in \BN^\Hu$. As explained in Subsection \ref{sec.reflection}, $\chi$ extends to 
an anti-involution $\hchi:\XD \to \XD$ so that $\hchi(x)=\chi(x)$ for $x\in R[\Hu]$ and $\hchi(X^\bk)= X^\bk$.

\subsection{Embedding of $\cS$ in the quantum torus $\sX(\D)$}\lbl{sec.torus}

 The following extends a result of 
 Muller \cite[Theorem 6.14]{Muller} for {\em totally} marked surfaces to the case of marked surfaces.

\begin{theorem}\lbl{r.torus} Suppose the marked surface $\SM$ has a quasitriangulation $\D$.

(a) There is a unique $R[\Hu]$-algebra embedding  $\vpD: \cS \embed \sX(\D)$ such that for all $a \in \D$, 
\be 
\vpD(a) = X_a.   \lbl{eq.vpD}
\ee 

(b) If we identify $\cS$ with its image under $\vpD$ then $\cS$ is sandwiched between $\sX_+(\D)$ and $\sX(\D)$, i.e.
\be 
\sX_+(\D) \subset \cS \subset \XD. \lbl{eq.incl2}
\ee
Consequently $\cS$ is a two-sided Ore domain, and  $\vpD$ induces an $\Rb$-algebra isomorphism 
$$ 
\tvpD : \widetilde \cS\SM \overset \cong \longrightarrow \widetilde \sX(\D),
$$
 where  $\widetilde \cS\SM$ and $ \widetilde \sX(\D)$ are the division algebras of $\cS\SM$ and $\sX(\D)$, respectively.

(c) Furthermore, $\vpD$ is reflection invariant, i.e. $\vpD$ commutes with $\hchi$.

\end{theorem}

\def\tf{\tilde f}

\begin{proof} The proof is a modification of Muller's proof for the totally marked surface case of Muller. To further simplify the proof, we will use Muller's result for the totally marked surface case.

(a) We first prove a few lemmas.   
\begin{lemma} The ring $\cS$ is a domain.
\end{lemma} 
\begin{proof} Let $\cP'\supset \cP$ be a larger set of marked points such that $(\Sigma,\cP')$ is totally marked. By Proposition \ref{r.func} $\cS\SM$ embeds into $\cS (\Sigma,\cP')$ which is a domain by Muller's result. Hence  $\cS$ is a domain.
\end{proof}

Relation \eqref{eq.35} shows that there is a unique $\Rb$-algebra homomorphism $\vDp: \sX_+(\D) \to \cS$ defined by $\vDp(X_a)=a$. Then for $\bn \in \BZ^\D$,
$$ \vDp(X^\bn)= \D^\bn := \left [  \prod_{a\in \D} a ^{\bn(a)} \right].$$

Note that $\vDp$ is injective since $\vDp$ maps the preferred $R$-basis $\BB_{\D,+}$ of $\sX_+(\D)$ bijectively onto a subset of the preferred $R$-basis $\BB_{\SM}$ of $\cS$. We will identify $\sX_+(\D)$ with its image under $\vDp$. Given a subset  $S\subset \D$, an {\em $S$-monomial} is an element in $\cS$ of the form $\D^\bn$, where $\bn \in \BN^\D$ has that $\bn(a)=0$ if $a\not \in S$.


Recall that $\D=\D_\inn \cup \D_\bd$, where $\D_\inn$ is the set of inner edges and $\D_\bd$ is the set of boundary edges.
\begin{lemma}  \lbl{r.monomial} 
Let $x\in \cS$.

(i) If $S\subset \D$ there is an $S$-monomial $\fm$ such that $\mu(a, x \fm)=0$ for all $a\in S$.

(ii) There is an $\D_\inn$-monomial $\fm$ such that $ x\fm \in \sX_+(\D)$.
\end{lemma}
\begin{proof} (i) The following two facts are respectively \cite[Lemma 4.7(3)]{Muller}  and \cite[Corollary 4.13]{Muller}:
\begin{align}
&\mu(a, yz) \le \mu(a,y) + \mu(a,z) \quad \text{for all} \ a\in \D,\ y,z\in \cS
\lbl{eq.mu1}\\
&\mu(a, y\,  a^{\mu(a,y)})=0  \quad \text{for all} \ a\in \D,\ y\in \cS. \lbl{eq.mu2}
\end{align}
These results are formulated and proved for general marked surfaces in \cite{Muller}, not just totally marked surfaces. Besides, since any two edges in $\D$ have intersection index 0,  we have
\be 
\mu(a, \fm) = 0 \quad \text{for all $\D$-monomials $\fm$ and  all} \ a\in \D. \lbl{eq.mu3}
\ee

Let $\bn\in \BZ^\D$ be given by $\bn(a)= \mu(a,x)$ for $a\in S$ and $\bn(a)= 0$ for $a\not\in S$. 
Then $\fm = \D^\bn$ is an $S$-monomial. Suppose $a\in S$.
By taking out the factors $a$ in $\fm$ and using~\eqref{e.normalizedtorus}, we have
$$ \fm = q^{k/2} a^{\mu(a,x)} \, \fm'$$
where $\fm'$ is another $S$-monomial and $k\in \BZ$. Using \eqref{eq.mu2} and then \eqref{eq.mu1} and \eqref{eq.mu3}, we have
$$ \mu(a, x \fm) \le \mu(a, x a^{\mu(a,x)}) + \mu (a, \fm') =0,$$
which proves $\mu(a, x \fm)=0$ for all $a\in S$.

(ii) Choose $\fm$ of part (a) with $S=\D_\inn$. Let $x_i \in \supp(x\fm)$. Clearly $\mu(a,x_i)=0$ for all $a \in \D_\bd$.
 Since $\mu(a,x\fm)=0$ for all $a\in \D_\inn$, 
 one can find a $\cP$-tangle $x_i'$ which is $\cP$-isotopic to $x_i$ such that $x_i$ and $a$ are taut (see Lemma \ref{l.FHS}), so that $x_i'\cap a=\emptyset$ in $\Sigma \setminus \cP$ for each $a\in \D$. The maximality in the definition of quasitriangulation shows that each component of $x_i'$ is $\cP$-isotopic to one in $\tD$. It follows that $x_i=\Hu^\bk \D^\bn$ in $\cS$ for some $\bk\in \BN^\Hu$ and $\bn \in \BN^\D$. This implies $x\fm \in \sX_+(\D)$.
\end{proof}

\begin{lemma}\lbl{l.oresubset} The multiplicative set
$\fM$ generated by  $\D$-monomials is a 2-sided Ore subset of $\cS$. 
Similarly the multiplicative set
$\fM_\inn$ generated by  $\D_\inn$-monomials is a 2-sided Ore subset of $\cS$.
\end{lemma}
\begin{proof}

 By  definition, $\fM$ is right Ore if
for every  $x\in \cS$ and every $ u \in \fM$, one has $x \fM \cap u\cS \neq \emptyset$.

 By \eqref{e.normalizedtorus}, one has $u=q^{k/2} \D^\bn$ for some $k\in \BZ$, $\bn\in \BN^\D$. 
By Lemma \ref{r.monomial}, there is a $\D$-monomial $\fm$ such that $x\fm \in \sX_+(\D)$. Since $\BB_{\D,+}=\{ \Hu^\bk \D^\bn \mid \bk \in \BN^\Hu, \bn \in \BN^\D \}$ is the preferred $R$-basis of $\sX_+(\D)$, we have a finite sum presentation $x \fm= \sum_{i\in I} c_i  \Hu^{\bk_i} \D^{\bn_i}$ with $ c_i \in R$. It follows that
\begin{align*}
 x \fM \ni x \fm u &= q^{k/2} \sum_{i\in I} c_i \Hu^{\bk_i} \D^{\bn_i} \D^\bn= q^{k/2} \sum_{i\in I} c_i q^{\la \bn_i, \bn\ra_P}\Hu^{\bk_i} \D^\bn \D^{\bn_i}   \\
 &= q^{k/2}  \D^\bn \sum_{i\in I} c_i q^{\la \bn_i, \bn\ra_P}\Hu^{\bk_i} \D^{\bn_i} = u \sum_{i\in I} c_i q^{\la \bn_i, \bn\ra_P}\Hu^{\bk_i} \D^{\bn_i} \in u \cS,
\end{align*}
where the second equality follows from  \eqref{e.normalizedtorus}. This proves $\fM$ is right Ore. Since the reflection anti-involution $\hchi$ reverses the order of the multiplication and fixes each $\D$-monomial, $\fM$ is also left Ore. The proof that $\fM_\inn$ is Ore is identical, replacing $\fM$ by $\fM_\inn$ everywhere.
\end{proof}
Let us prove Theorem \ref{r.torus}(a).  Since $\cS$ is a domain, the natural map $\cS \to \cS\fM^{-1}$, where $\cS\fM^{-1}$ is the Ore localization of $\cS$ at $\fM$, is injective.  Since Ore localization is flat, the inclusion $\vDp :\sX_+(\D) \embed \cS$ induces  an inclusion
 \be
 \lbl{eq.iso1}
  \vDpt:\sX_+(\D) \fM^{-1} \embed  \cS\fM^{-1}.
 \ee
 Note that $\sX_+(\D) \fM^{-1}= \sX(\D)$.
Let us prove $\vDpt$ is surjective.  After identifying $\sX_+(\D) \fM^{-1}$ as a subset of $\cS\fM^{-1}$ via $\vDpt$, it is enough to show that $\cS\subset \sX_+(\D) \fM^{-1}$.  This is guaranteed by Lemma \ref{r.monomial}, thus $\vDpt$ is an isomorphism.

 Let $\vD$ be the restriction of $(\vDpt)^{-1}$ onto $\cS$, then we have an embedding of $R[\Hu]$-algebras $\vD : \cS \embed \sX(\D)$ such that $\vD\circ \vDp$ is the identity on $\sX_+(\D)$. Any element  $x\in\cS$ can be presented as $y \fm^{-1}$ with $y \in \sX_+(\D), \fm \in \fM$. This shows $\sX_+(\D)$ weakly generates $\cS$, and thus the uniqueness of $\vpD$ is clear.
 This proves (a).
 
 (b) Inclusion \eqref{eq.incl2} follows from \eqref{eq.vpD}, and part (b) follows from Corollary \ref{r.sandwich}.

(c)  Let us prove that $\vD$ is reflection invariant, i.e. for every $x \in \cS$ , we have
  \be
  \lbl{eq.4d}
  \vD(\hchi(x))= \hchi(\vD(x)).
  \ee
  Identity \eqref{eq.4d} clearly holds for the case when $x \in \tD$. Hence it holds for $x$ in the $R$-algebra generated by $\tD$, which is $\sX_+(\D) \subset \cS$. Since every element $x \in \cS$ can be presented in the form $y \fm^{-1}$, where $y \in \sX_+(\D), \fm \in \fM$, we also have \eqref{eq.4d} for $x$. This completes the proof of Theorem~\ref{r.torus}. \end{proof}

\def\ttS{\tilde{\mathcal S}}
\def\La{{\Lambda}}
\def\cD{\mathcal D}
\def\cC{\mathcal C}
\begin{remark}\lbl{rem.Deltain} Lemma \ref{r.monomial} shows that $\vpD(\cS)$ lies in $\sX_+(\D) (\fM_\inn) ^{-1}$.
\end{remark}

\subsection{Flip and transfer} For a quasitriangulation $\D$ of $\SM$, 
 the map $\vD: \cS\embed \sX(\D)$ will be called the {\em skein coordinate map}. We wish to understand how these coordinates change under change of quasitriangulation.
 
Let us first introduce the notion of a flip of a $\cP$-quasitriangulation.
\FIGc{mutation}{Flip $a \to a^*$. Case 1}{3cm}
\FIGc{quasiflip}{Flip $a \to a^*$. Case 2}{2.5cm}

Suppose $\D$ is a quasitriangulation of $\SM$ and $a$ is an inner edge in $\D$. The {\em flip of $\D$ at $a$} is a new quasitriangulation given by $\D' := \D \setminus \{a\} \cup \{ a^*\}$, where $a^*$ is the only $\cP$-arc not $\cP$-isotopic to $a$ such that $\D'$ is a quasitriangulation.
There are two cases:
\begin{itemize}
\item[\em Case 1.] $a$ is the common edge of two distinct triangles, see Figure \ref{fig:mutation}.
\item[\em Case 2.] $a$ is the common edge of a holed monogon and a triangle, see Figure \ref{fig:quasiflip}.
\end{itemize}
 In both cases 
 $a^*$ is depicted in Figures \ref{fig:mutation} and \ref{fig:quasiflip}.

Any two $\cP$-quasitriangulations are related by a sequence of flips, see e.g. \cite{Penner}, where a flip of case 2 is called a quasi-flip. If $\SM$ is a totally marked surface, then there is no flip of case 2. 

 Suppose $\D, \D'$ are two quasitriangulations of $(\Sigma, \cP)$. Let
 $$\Theta_{\D,\D'}:= \tilde{\vp}_{\D'} \circ (\tvpD)^{-1}:  \tiX(\D)  \to \tiX(\D').$$
  By Theorem \ref{r.torus}, $\Theta_{\D,\D'}$ is an $\Rb$-algebra isomorphism from $ \tiX(\D)$  onto $\tiX(\D')$. We call $\Theta_{\D,\D'}$ the {\em transfer isomorphism from $\D$ to $\D'$}.

\begin{proposition} \lbl{r.42}

(a) The transfer isomorphism $\Theta_{\D,\D'}$ is natural. This means
 $$\Theta_{\D, \D}= \Id, \quad \Theta_{\D, \D''} = \Theta_{\D', \D''} \circ \Theta_{\D, \D'}.$$

 (b) The skein coordinate maps $\vp_\D: \cS \embed \tiX(\D)$ commute with the transfer isomorphisms, i.e.
$$ \vp_{\D'} = \Theta_{\D,\D'} \circ \vp_{\D}.$$

 (c) Suppose $\D'$ is obtained from $\D$ by a flip at an edge  $a$, with $a$ replaced by $a^*$ as in  Figure \ref{fig:mutation} (Case 1) or Figure \ref{fig:quasiflip} (Case 2).
Identify $\cS$ as a subset of $\XD$ and $\tiX(\D')$. Then, with notations of edges as in Figure \ref{fig:mutation}  or Figure \ref{fig:quasiflip}, we have 
\begin{align}
 \lbl{eq.aa0} \Theta_{\Delta\Delta'}(u)&= u \quad \text{ for} \ u\in \D \setminus \{a\},\\
\lbl{eq.aa1}
 \Theta_{\Delta\Delta'}(a) &=
 \begin{cases}
 [ce(a^*)^{-1}] + [ bd(a^*)^{-1}] \quad &\text{in Case 1}\\
 [b^2 (a^*)^{-1}] + [c^2 (a^*)^{-1}] + \beta [bc(a^*)^{-1}] &\text{in Case 2}.
 \end{cases}
\end{align}

\end{proposition}
\begin{proof}
Parts (a) and (b) follow right away from the definition. Identity \eqref{eq.aa0} is obvious from the definition. 
Case 1 of \eqref{eq.aa1} is proven in  \cite[Proposition 5.4]{Le3}.

For case 2 of \eqref{eq.aa1}, we have that $\Theta_{\Delta,\Delta'}(a) = \tilde{\vp}_{\D'}(a)$. To compute this, we note that in $\cS$, $aa^*=q^2b^2+q^{-2}c^2+\beta bc$. Then

\begin{align*}
\tilde{\vp}_{\D'}(aa^*) & = \tilde{\vp}_{\D'}(q^2b^2) + \tilde{\vp}_{\D'}(q^{-2}c^2) + \tilde{\vp}_{\D'}(\beta bc), \\
\tilde{\vp}_{\D'}(a)\tilde{\vp}_{\D'}(a^*) &= q^2b^2+q^{-2}c^2+\beta bc, \\
\tilde{\vp}_{\D'}(a)a^* &=  q^2b^2+q^{-2}c^2+\beta bc
\end{align*}
Multiply both sides on the right by $(a^*)^{-1}$ and note that the $q$ factors agree with Weyl normalization. Therefore,

$$
\tilde{\vp}_{\D'}(a) = q^2b^2(a^*)^{-1} + q^{-2}c^2(a^*)^{-1} + \beta bc(a^*)^{-1} = [b^2(a^*)^{-1}] + [c^2(a^*)^{-1}] + \beta [bc(a^*)^{-1}].
$$
\end{proof}

\def\rdpp{\rho_{\D',\D''}}
\def\mon{\text{mon}}
\section{Modifying marked surfaces} \lbl{sec.surgery}

 In this section, a quasitriangulable marked surface $\SM$ is fixed and we write $\cS:=\cS\SM$.
The inclusion of unmarked boundary components in the theory allows us to describe how the skein coordinates change under modifications of surfaces. In this section we
 consider two modifications: adding a marked point  and plugging an unmarked boundary component with a disk. The results of this section will be used in the proof of the main theorem, particularly for Proposition \ref{r.knot}.

\def\TDDs{\theta_{\D,\D}^{\text{sur}}}
\def\Ts{\theta^{\text{sur}}}
\def\BDs{B_\D^{\text{sur}}}
\def\BDps{B_{\D,+}^{\text{sur}}}
\subsection{Surgery algebra}\lbl{sec.suralg}Let $\D$ be a quasitriangulation of $\SM$. 
Identify $\cS$ as a subset of $\XD$ using the skein coordinate map $\vD$. 
Recall that $\D_\mon$ is the set of all monogon edges. Let $\Dess:= \D \setminus \D_\mon$.

Suppose $\beta\in \Hu$ is an unmarked component whose monogon edge is $a_\beta$. Let $\Sigma'$ be the result of gluing a disk to $\Sigma$ along $\beta$. We will say that $(\Sigma', \cP)$ is obtained from $\SM$ by {\em plugging 
 the unmarked $\beta$}.   Then $a_\beta$ becomes 0 in $\cS(\Sigma', \cP)$, while it is invertible in $\XD$. For this reason we want to find a subalgebra $\ZD$ of $\XD$ in which $a_\beta$ is not invertible, but we still have $\cS \subset \ZD$.

For $a\in \D_\mon$ choose a $\cP$-arc $a^*$ such that $\D\setminus \{a\} \cup \{a^*\}$ is a new quasitriangulation, i.e. it is the result of the flip of $\D$ at $a$. Let $\D_\mon^*:= \{ a^* \mid a \in \D_\mon\}$. 
For $a\in \D_\mon$ let $(a^*)^*=a$.

The {\em surgery algebra}  $\CZ(\D)$ is the $R[\Hu]$-subalgebra of $\sX(\D)$ generated by $a^{\pm 1}$ with $a\in \Dess$, and  all $a\in  \D_\mon \cup \D_\mon^*$. Thus in $\ZD$ we don't have $a^{-1}$ (for  $a\in \D_\mon$) but we do have $a^*$, which will suffice in many applications. With the intention to replace $a^{-1}$ by $a^*$, we introduce the following definition: for $a\in \D$ and $k\in \BZ$ define
$$ a^{\{k\}} = \begin{cases} (a^*)^{-k} \qquad & \text{if  $a\in \D_\mon$ and $k <0$ }\\
a^k & \text{in all other cases.}\end{cases}
$$
For $\bk\in \BZ^\D$ define
$$ \D^{\{\bk\}} := \left [ \prod_{a\in \D} a^{\{\bk(a)\}}   \right].$$

\begin{proposition}\lbl{p.ztox}

(a)  As an $R[\Hu]$-algebra, $\ZD$ is 
generated by $\D\cup \D_\mon^*$ and $a^{-1}$ for $a\in
\Dess$,  
subject to the following relations: 
 \begin{align}
xy &= q^{P(x,y)}yx \quad &\text{if}& \  (x,y) \neq (a, a^*) \ \text{for all } \ a\in (\D_\mon \cup \D_\mon^*)   \lbl{eq.PP} \\
a a^*& = q^2b^2+q^{-2}c^2+\beta bc, &\text{if}& \ a\in (\D_\mon \cup \D_\mon^*).
 \lbl{e.zrelations}
 \end{align}
Here, for  the case where $a \in \D_\mon \cup \D_\mon^*$, we denote by $\beta$  the unmarked boundary component surrounded by $a$, and the edges $b,c$  are  as in Figure \ref{fig:quasiflip}.

(b) The skein algebra $\cS$ is a subset of $\ZD$ for any quasitriangulation $\D$.

(c) As an $R[\Hu]$-module, $\ZD$ is free with basis
$$\BDs:= \{ \D^{\{\bk\}} \mid \bk \in \BZ^\D\}.$$
\end{proposition}

It should be noted that $P(x,y)$  in \eqref{eq.PP} is well-defined  since $x,y$ do not intersect in $\Sigma \setminus \cP$.

\def\fN{\mathfrak{N}}

\begin{proof} (a) Let us redefine $\ZD$ so that it has the presentation given 
 in the proposition. Recall that $\cS\subset \XD$, hence $a^*\in \XD$. Define an  $R[\Hu]$-algebra homomorphism $\theta: \ZD \to \XD$ by $\theta(a)= a$ for all $a\in \D \cup \D_\mon^*$. Since all the defining relations of $\ZD$ are satisfied in $\XD$, the homomorphism $\theta$ is well-defined. To prove (a) we need   to show that $\theta$ is injective. 

 Let $\CZ_+(\D) \subset \CZ(\D)$ be the $R[\Hu]$-algebra generated by all $a\in \D \cup \D_\mon^*$.

\begin{lemma}\lbl{r.sinczp}
Let $T \subset \cS$ be an essential $\cP$-tangle such that $\mu(T,b)=0$ for all $b \in \Dess$. Then $T \in \CZ_+(\D)$.
\end{lemma}

\begin{proof} After a $\cP$-isotopy we can assume $T$ does not intersect any $b\in \Dess$ in $\Sigma \setminus \cP$ by Lemma \ref{l.FHS}.
Cutting $\Sigma$ along $\Dess$, one gets a collection of ideal triangles and {\em eyes}. Here  an {\em eye} is a bigon with a small open disk removed, see Figure \ref{fig:eye}. Each eye has  two  quasitriangulations. 
If $x$ is a component of $T$, then $x$, lying inside of a triangle or an eye, must be $\cP$-isotopic to an element in $\D \cup \D_\mon^* \cup \Hu$, which implies $x \in \CZ_+(\D)$. Hence $T \in \CZ_+(\D)$.
\end{proof}
 
\FIGc{eye}{An {\em eye} (left), and its two quasitriangulations.}{2cm}

\begin{lemma}\lbl{l.surgerybasis} 
(i) The  set 
$$
\BDps := \{ \D^{\{\bk\}} \mid \bk \in \BN^{\Dess} \times \BZ^{\D_{\mon}}\}.
$$
is an $R[\Hu]$-basis for $\CZ_+(\D)$.
The map $\theta$ maps $\CZ_+(\D)$ injectively into $\cS$.

(ii)  The multiplicative set $\fN$ generated by all $\Dess$-monomials is a two-sided Ore set of $\CZ_+(\D)$ and $\ZD= \CZ_+(\D) \fN^{-1}$.
\end{lemma}

\begin{proof} 
(i) Clearly $\CZ_+(\D)$ is $R[\Hu]$-spanned by monomials in elements of $\D \cup \D_\mon^*$. 
Since each $b\in \Dess$ will $q$-commute with any element of $\D \cup \D_\mon^*$, 
every monomial in elements of $\D \cup \D_\mon^*$ is equal to, up to a factor which is a power of $q$, an element of the form
\be
\lbl{eq.1q}
x= a_1 \dots a_l \,  \Dess^\bk, \quad \bk\in \BN^{\Dess},
\ee
where $a_i \in \D_\mon\cup \D_\mon^*$ and $\Dess^\bk$ is understood to be a $\Dess$-monomial. Each $a \in \D_\mon\cup \D_\mon^*$ commutes with every element of $\D_\mon\cup \D_\mon^*$ except for $a^*$. 
If for any $a\in \D_\mon \cup \D_\mon^*$, $\{a,a^*\} \not\subset \{a_1,\dots, a_l\}$, then  $x\in \BDps$, up to a factor which is a power of $q$.

If $\{a,a^*\} \subset \{a_1,\dots, a_l\}$, then we can permute the product to bring one $a$ next to one $a^*$, and relation \eqref{e.zrelations} shows that $x$ is equal to an $R[\Hu]$-linear combination of elements of the form \eqref{eq.1q} each of which have a smaller number of $a,a^*$. Induction shows that elements $x$ of the form \eqref{eq.1q} are linear combinations of elements of the same form \eqref{eq.1q} in which not both $a$ and $a^*$ appear for every $a\in \D_\mon \cup \D_\mon^*$. This shows $\BDps$ spans $\CZ_+(\D)$ as an $R[\Hu]$-module.

The geometric realization of elements in $\BDps$ (i.e. their image in $\cS$) shows that $\theta$ maps $\BDps$ injectively into the preferred basis $B_{\SM}$ of $\cS$.
 This shows that $B_{\CZ,+}$ must be $R[\Hu]$-linearly independent, that $B_{\CZ,+}$ is an $R[\Hu]$-basis of $\CZ_+(\D)$, and that $f$ maps $\CZ_+(\D)$ injectively into $\cS$.
 
(ii) As $\CZ_+(\D)$ embeds in $\cS$, which is a domain, $\fN$ contains only regular elements.
To get $\ZD$ from $\CZ_+(\D)$ we need to invert all $a\in \Dess$. 
 As every $a\in \Dess$ will $q$-commute with any other generator, 
every element  of $\ZD$ has the form $x \fn^{-1}$ and also the form $(\fn')^{-1} x'$, where $x,x'\in\CZ_+(\D)$ and $\fn, \fn'$ are $\Dess$-monomials. Thus $\ZD$ is a  ring of fractions of $\CZ_+(\D)$ with respect to $\fN$. This shows that $\fN$ is a two-sided Ore set of $\CZ_+(\D)$ and $\CZ(\D) = \CZ_+(\D) \fN^{-1}$.
\end{proof}

Suppose $\theta(x)=0$ where $x\in \CZ$. Then $x=y \fn^{-1}$ for some $y\in \CZ_+(\D)$ and  $\fn\in \fN$. Hence $\theta(y) = \theta(x) \theta(\fn)=0$. 
 Lemma \ref{l.surgerybasis} shows $y=0$. Consequently $x=0$. This proves the injectivity of~$\theta$.

(b) Suppose $x \in \cS$. Using Lemma \ref{r.monomial} with $S=\Dess$, there is a $\Dess$-monomial $\mathfrak{n}$ such that   $\mu(a, x \mathfrak{n}) =0$ for all $a\in \Dess$. Then $x \mathfrak{n}$ is an $R$-linear combination of essential $\cP$-tangles which do not intersect any edge in $\Dess$ by Lemma \ref{l.FHS}. By Lemma \ref{r.sinczp}, it follows that $x \mathfrak{n} \in  \CZ_+(\D)$.
Hence $x = (x \mathfrak{n}) \mathfrak{n}^{-1} \in \CZ_+(\D)  \fN^{-1}  = \ZD$. This proves $\cS \subset \ZD$.

(c) As any element of $\ZD$ may be written as $x \fn^{-1}$, where $x\in\CZ_+(\D)$ and $\fn$ is a $\Dess$-monomial, and $\BDps$ spans $\CZ_+(\D)$ as an $R[\Hu]$-module, we have that $\BDs$ spans $\CZ(\D)$ as an $R[\Hu]$-module. On the other hand, suppose we have a  $R[\Hu]$-linear combination of
$\BDs$ giving 0:
$$ \sum_i c_i\,  \D^{\{\bk_i\}} =0, \quad c_i \in R[\Hu].$$
Multipliying on the right by $(\Dess)^\bk$ where $\bk(a)$ is sufficiently large for each $a\in \Dess$, we get
$$ \sum_i q^{l_i/2}\, c_i \, \D^{\{\bk'_i\}} =0, \quad l_i \in \BZ,$$
where $\bk'_i(a) \ge 0$ for all $a\in \Dess$. This means each $\D^{\{\bk'_i\}}$ is in $\BDps$, an $R[\Hu]$-basis of $\CZ_+(\D)$. It follows that $c_i=0$ for all $i$. Hence, $\BDs$ is 
linearly independent over 
$R[\Hu]$, and consequently an $R[\Hu]$-basis of $\ZD$.
\end{proof}

\begin{lemma}\lbl{l.skeinsurgeryextension} An $R$-algebra homomorphism
$g: \cS \to A$  extends to an $R$-algebra homomorphism $\CZ(\D) \to A$ if and only if $g(a)$ is invertible for all $a \in \Dess$, and furthermore the extension is unique.

\end{lemma}

\begin{proof} We have $\CZ_+(\D)\subset \cS\subset \CZ(\D)= \CZ_+(\D) \fN^{-1}$. By Proposition \ref{r.sandwich0}, $\fN$ is a two-sided Ore subset of $\cS$ and $\ZD= \cS \fN^{-1}$. The lemma follows from the universality of the Ore extension. 
\end{proof}

\subsection{Adding marked points}\lbl{sec.addingmarkedpoints}
Let $p\in \pS \setminus \cP$, and $\cP'= \cP \cup \{ p\}$. The identity map $\iota: \Sigma\to \Sigma$ induces an $R$-algebra embedding $\iota_*:\cS\SM \to \cS(\Sigma,\cP')$, see Proposition \ref{r.func}. After choosing how to extend a $\cP$-quasitriangulation $\D$ to be a $\cP'$-quasitriangulation $\D'$, we will show  that $\iota_*$ has a unique extension to an $R$-algebra homomorphism $\Psi: \CZ(\D) \to \CZ(\D')$ which makes the following diagram commute. The map $\Psi$ describes how the skein coordinates change.

\be \lbl{eq.cd1}
\begin{tikzcd}
\cS(\Sigma,\cP) \arrow[d,"\iota_*"] \arrow[r,hook,"\vpD"] & \CZ(\D) \arrow[d,"\Psi"] \\
\cS(\Sigma,\cP') \arrow[r,hook,"\vpDp"] & \CZ(\D')
\end{tikzcd}
\ee

There are two scenarios to consider: adding a marked point to an unmarked boundary component or to a boundary edge.


\def\hD{\widehat{\D}}

\def\SMp{(\Sigma, \cP')}
\subsection{Scenario 1: Adding a marked point to a boundary edge}\lbl{sec.addmarkedptboundary}
Suppose $a\subset \pS$ is a boundary $\cP$-arc of $\SM$ and $p$ is a point in the interior of $a$.
Let $\cP'=\cP \cup \{p\}$.  The set $\Hu'$  of unmarked boundary components of the new marked surface $(\Sigma,\cP')$ is equal to~$\Hu$.

Let $\D$ be a $\cP$-quasitriangulation of $\SM$. 
 Define a $\cP'$-quasitriangulation $\D'$ of $\Sigma$ by $\D':=\D \cup \{a_1,a_2\}$ as shown in Figure \ref{fig:boundarypoint} (where we have $\cP'$-isotoped $a$ away from $\partial \Sigma$ in $\D'$).

\FIGc{boundarypoint}{Adding a marked point to a boundary edge. The shaded part is a subset of the interior of $\Sigma$.}{2cm}

Recall that $\CZ(\D)$ is weakly generated by $\D \cup \D_\mon^*$, and that $\Dess= \D\setminus \D_\mon$.

\begin{proposition}\lbl{prop.addptsurgerytorus} There exists a unique $R[\Hu]$-algebra homomorphism $\Psi: \CZ(\D) \to \CZ(\D')$ which makes diagram \eqref{eq.cd1} commutative. Moreover, $\Psi$ is given by  $\Psi(a)=a$ for all $a \in \D \cup \D_\mon^*$.
\end{proposition}
\begin{proof} Identify $\cS\SM$ with its image under $\vpD$ and $\cS\SMp$ with its image under $\vpDp$.
Note that $\Dess \subset \D'_{\mathrm{ess}}$. Hence if $a \in \Dess$ then $\iota_*(a)=a$ is invertible in $\CZ(\D')$.
That 
$\Psi$ exists uniquely follows from Lemma \ref{l.skeinsurgeryextension} because $\vpDp\circ \iota_*(a)$ is invertible for all $a \in \Dess$. That $\Psi(a) = a$ for all $a \in \D \cup \D_\mon^*$ follows immediately.
\end{proof}

\subsection{Scenario 2: Adding a marked point to an unmarked component} Suppose $\beta\in \Hu$ is an unmarked boundary component of $\SM$. Choose a point $p \in \beta$ and set $\cP' = \cP \cup \{p\}$. We call the new marked surface $(\Sigma, \cP')$ and write $\Hu'= \Hu \setminus \beta$ for its set of unmarked boundary components. 

Suppose $\D$ is a quasitriangulation of $\SM$. 
Let $a \in \D_\mon$ be the monogon edge bounding the eye containing the unmarked boundary component $\beta$ (as defined in Figure \ref{fig:eye}), and $b,c \in \D$ be the edges immediately clockwise and counterclockwise to $a$ as depicted on the left in Figure \ref{fig:holepoint}. To get a triangulation of the eye containing $\beta$ with the added marked point $p$, we need to add 3 edges $d,e,f$ as depicted on the right side of Figure \ref{fig:holepoint}. Here 
 $f$ is the boundary $\cP'$-arc whose ends are both $p$. By relabeling, we can assume that $e$ is counterclockwise to $d$ at $p$.
 Up to isotopy of $\Sigma$ fixing every point in the complement of the  monogon, there is only one choice for such $d$ and $e$. Then $\D'=\D \cup \{d,e,f\}$ is a quasitriangulation of $(\Sigma,\cP')$.

\FIGc{holepoint}{From $\D$ to $\D'$.}{3.1cm}


Since $\Hu\neq \Hu'$,  it is not appropriate to consider modules and algebras over $R[\Hu]$. 
Rather we will consider both $\cS\SM$ and $\cS\SMp$ as algebras over $R$. As an $R$-algebra, $\CZ(\D)$ is weakly generated by $\D \cup \D_\mon^*\cup \Hu$.

\begin{proposition}\lbl{lem.markedpthole} There exists a unique $R$-algebra homomorphism $\Psi: \CZ(\D) \to \CZ(\D')$ which makes diagram \eqref{eq.cd1} commutative. Moreover, for
$z \in \D \cup \D_\mon^*\cup \Hu$ we have
\be \lbl{eq.1f}
\Psi(z) = 
\begin{cases} z  \quad &\text{if } \ z \notin \{\beta,a^*\}, \\
[d^{-1}e]+[ad^{-1}e^{-1}f]+[de^{-1}]   &\text{if } \ z = \beta,\\ 
[a^{-1}b^2] + [a^{-1}c^2] + [a^{-1}bcd^{-1}e] + [bcd^{-1}e^{-1}f] + [a^{-1}bcde^{-1}] &\text{if } \ z = a^*.
\end{cases} 
\ee
\end{proposition}

\begin{proof} Again $\Dess\subset \D'_{\mathrm{ess}}$, so the existence and uniqueness of $\Psi$ follows.
Formula  \eqref{eq.1f} follows from a simple calculation of the value of $\iota_*(z)$.
\end{proof}


\def\SpM{(\Sigma',\cP)}
\subsection{Plugging a hole} The more interesting operation is plugging a hole.

Fix an unmarked boundary component $\beta \in \Hu$. Let $\Sigma'$ be the result of gluing a disk to $\Sigma$ along $\beta$. Then $(\Sigma', \cP)$ is another marked surface. The natural morphism
$ \iota: \SM \to \SpM$ gives rise to an $R$-algebra homomorphism $\iota_*: \cS\SM\to \cS\SpM$. Since $\iota_*$ maps the preferred $R$-basis $B_{\SM}$ of $\cS\SM$ onto a set containing the preferred $R$-basis $B_{(\Sigma',\cP)}$ of $\cS\SpM$, the map $\iota_*$ is surjective.
\FIGc{surgery}{From $\D$ to $\D'$}{2.3cm}

Suppose $\D$ is a quasitriangulation of $\SM$.
Let $a\in \D$ be the monogon edge bounding the eye containing the unmarked boundary $\beta$ and $\tau$ be the triangle having $a$ as an edge. Let $a,b,c$ be the edges of $\tau$ in counterclockwise order, as in Figure \ref{fig:surgery}. Let $\D'= \D \setminus \{a,b\}$. Then $\D'$ is a $\cP$-quasitriangulation of $\Sigma'$.

We cannot extend $\iota_*: \cS\to \cS'$ to an $R$-algebra homomorphism  $ \sX(\D) \to \sX(\D')$, since $\iota_*(a)=0$ but $a$ is invertible in $\XD$. 
This is the reason why we choose to work with the smaller algebra $\CZ(\D)$ which does not contain $a^{-1}$.


Recall that as an $R$-algebra, $\ZD$ is weakly generated by $\Hu \cup \D_\mon^* \cup \D$.

\def\bPsi{\bar \Psi}
\def\bX{\bar X}
\begin{proposition}\lbl{t.holetrick} There exists a unique $R$-algebra homomorphism 
$\Psi: \CZ(\D) \to \CZ(\D')$ such that the following diagram
\be\lbl{d.holesurgery}
\begin{tikzcd}
\cS(\Sigma,\cP) \arrow[d,"\iota_*"] \arrow[r,hook,"\vpD"] & \CZ(\D) \arrow[d,"\Psi"] \\
\cS(\Sigma',\cP) \arrow[r,hook,"\vpDp"] & \CZ(\D')
\end{tikzcd}
\ee
is commutative. Explicitly, $\Psi$ is defined on the generators in $\Hu \cup \D_\mon^* \cup \D$ as follows: 
\begin{align}
\Psi(e)& = e  \quad \text{if} \ e  \in (\Hu \cup \D_\mon^* \cup \D)\setminus \{a,a^*, b, \beta\}
\lbl{eq.2a}\\
\Psi(a)&= \Psi(a^*)=0, \ 
 \Psi(b)=c, \ 
 \Psi(\beta)= -q^2 - q^{-2}.
 \lbl{eq.2b}
\end{align}
The map  $\Psi$ is surjective and its  kernel is  the ideal $I$ of $\ZD$ generated by $a, a^*, b-c, \beta + q^2 + q^{-2}$. 
\end{proposition}
\begin{proof} Identify $\cS\SM$ with its image under $\vpD$ and $\cS\SMp$ with its image under $\vpDp$.

The existence and uniqueness of 
 $\Psi$ follows from Lemma \ref{l.skeinsurgeryextension}, since $ \iota_*(x)$ is invertible in $\CZ(\D')$ for all $x \in \D \setminus \D_\mon$. By checking the value of $ \iota_*(x)$ for $x\in \Hu \cup \D_\mon^* \cup \D$, we get \eqref{eq.2a} and \eqref{eq.2b}. 
 It follows that $I$ is in the kernel $\ker\Psi$. Hence $\Psi$ descends to an $R$-algebra homomorphism
$$ \bPsi: \ZD/I \to \CZ(\D').$$
We will prove that $\bPsi$ is bijective by showing that there is  an $R$-basis $X'$ of $\CZ(\D')$ and an $R$-spanning set $\bX$  of $ \ZD/I$ such that $\bPsi$ maps $\bX$ bijectively onto $X'$. Then $\bPsi$ is an isomorphism. Let 
\begin{align*}
X &=  \{ \D^{\{\bk\}} \, (\Hu)^\bn \mid \bk \in \BZ^{\D}, \bn \in \BN^{\Hu}  \},   \\
X_0 &=  \{ \D^{\{\bk\}} \, (\Hu)^\bn \in X \mid \bk(a)= \bk(b)=\bn(\beta)=0   \} \subset X, \\
X' &=\{ (\D')^{\{\bk\}} \, (\Hu')^\bn \mid \bk \in \BZ^{\D'}, \bn \in \BN^{\Hu'}  \}.   
\end{align*}
By Proposition \ref{p.ztox}(c), the sets $X$ and $X'$ are respectively $R$-bases of $\CZ(\D)$ and $\CZ(\D')$. As $\D'= \D \setminus \{a,b\}$ and $\Hu'= \Hu\setminus \{\beta\}$, Formula \eqref{eq.2a} shows that $\Psi$ maps $X_0$ bijectively onto $X'$. Consequently, the projection $\pi: \ZD\to \ZD/I$ maps $X_0$ bijectively onto a set $\bX$ and
$\bPsi$ maps $\bX$ bijectively onto $X'$.

It remains to be shown that  the $R$-span $R\la \bX\ra$ of $\bX$ equals $\CZ(\D)/I$. Suppose $x= \D^{\{\bk\}} \Hu^\bn  \in \Psi(X) \setminus X_0$. 
Then either $\bk(a)\neq 0$, or $\bk(b)\neq 0$, or $\bn(\beta)\neq 0$.

If $\bk(a)\neq 0$, then in $x$ there is factor of $a$ or $a^*$ which is in $I$, and hence $\pi(x)=0$.
Because $b-c$ and $\beta + q^2 + q^{-2}$ are in $ I$, in $\CZ(\D)/I$ we can replace $b$ by $c$  and $\beta$ by the scalar $-q^2 - q^{-2}$. Thus, $\pi(x) \in R\la \bX\ra$. As $X$ spans $\CZ(\D)$, this shows $\bX$ spans $\CZ(\D)/I$. The proposition is proven.

\end{proof}

\def\Trq{\text{tr}_q}
\def\TrqN{\text{tr}_{q^{N^2}}}
\def\Trxi{\text{tr}_{\xi}}
\def\circD{\accentset{\circ}{\D}}
\def\An{\mathbb A}

\section{Chebyshev-Frobenius homomorphism}

For the case when the marked set is empty,  Bonahon and Wong  \cite{BW1} constructed a remarkable algebra homomorphism, called the Chebyshev homomorphism,  from the skein algebra  with quantum parameter $q=\xi^{N^2}$ to the skein algebra with quantum parameter $q=\xi$, where $\xi$ is a complex root of unity, and $N$ is the order (as a root of unity) of $\xi^4$.
In \cite{BW1} the proof of the existence of the Chebyshev homomorphism is based on the theory of the quantum trace map \cite{BW0}. Since the result can be formulated using only elementary skein theory, Bonahon and Wong asked for a skein theoretic proof of their results. Such a proof was offered in \cite{Le2}.

 Here we extend the result of Bonahon and Wong to the case of marked 3-manifolds. 
Our proof  is different  from the two above mentioned proofs even in the case of the marked set being empty; it does not rely on many computations but rather on the functoriality of the skein algebras.


\subsection{Setting}
Throughout this section we fix a marked 3-manifold  $(M,\cN)$. The ground ring $R$ is $\BC$. Let $\Cx$ denote the set of non-zero complex numbers. A {\em root of unity} is a complex number $\xi$ such that there exists a positive integer $n$ such that $\xi^n=1$, and the smallest such $n$ is called the order of $\xi$, denote by  $\ord(\xi)$.

 The skein module $\cS\MN$ depends on the choice of  $q=\xi \in \Cx$, and we denote the skein module with this choice by $\cS_\xi(M,\cN)$.  To be precise, we also  choose and fix one of the two square roots of $\xi$ for the value of $q^{1/2}$, but the choice of $\xi^{1/2}$ is not important due the symmetry of complex conjugation.

Similarly, if $\SM$ is a marked surface with quasitriangulation $\D$, then we use the notations $\cS_\xi\SM$, $\sX_\xi(\D), \CZ_\xi(\D)$ to denote what were respectively the $\cS\SM, \sX(\D), \CZ(\D)$ of Subsections \ref{sec.torus} and \ref{sec.suralg} with ground ring $\BC$ and $q=\xi$. We will always identify $\cS_\xi\SM$ as a subset of $\sX_\xi(\D)$, its division algebra $\widetilde {\sX}_\xi(\D)$, and $\CZ_\xi(\D)$.

\subsection{Formulation of the result}\lbl{sec.formulationmain} 
For $\xi\in \Cx$ recall that $\cS_\xi\MN= \cT\MN/\Rel_\xi$, where $\cT\MN$ is the  $\BC$-vector space with basis the set of all $\cN$-isotopy classes of $\cN$-tangles in $M$  and $\Rel_\xi$ is the subspace spanned by the trivial loop relation elements, the trivial arc relation elements, and the skein relation elements, see Subsection \ref{s.def00}. For $x\in \cT\MN$ denote by $[x]_\xi$ its image in $\cS_\xi\MN= \cT\MN/\Rel_\xi$.

 For  an $\cN$-arc or an $\cN$-knot  $T$ in  $(M,\cN)$ and $k \in \BN$ let $T^{(k)}$  be {\em the $k$th framed power of $T$}, which is $k$ parallel copies of $T$ obtained using the framing, which will be considered as an $\cN$-tangle lying in a small neighborhood of $T$. The $\cN$-isotopy class of $T^{(k)}$ depends only on the $\cN$-isotopy class of $T$.

Given a polynomial $P(z) = \sum c_i z^i \in \BZ[z]$, and an $\cN$-tangle $T$ with a single component we define an element $P^{\text{fr}}(T) \in \cT\MN$ called the {\em threading of $T$ by $P$} by $P^{\text{fr}}(T) = \sum c_i T^{(i)}$. If $T$ is $\cP$-knot in a marked surface $\SM$ then $P^{\text{fr}}(T)=P(T)$. Using the definition \eqref{eq.simul} one can easily check that  $P^{\text{fr}}(T)=P(T)$ for the case when $T$ is a $\cP$-arc as well.

Fix  $N\in \BN$. Suppose  $T_N(z)= \sum c_i z^i$ is the $N$th Chebyshev polynomial of type 1 defined by~\eqref{eq.Che}.
Define a $\BC$-linear map
$$ \hPhi_{N}: \LMN \to \LMN$$ so that  if $T$ is an $\cN$-tangle then  $\hPhi_N(T)\in \LMN$ is the union of $a^{(N)}$ and $(T_N)^{\text{fr}}(\al)$ for each $\cN$-arc component $a$ and each $\cN$-knot component $\al$ of $T$.  See Subsection \ref{sec.func} for the precise definition of union for skein modules. In other words, $\hPhi_N$ is given by threading each $\cN$-arc by $z^N$ and each $\cN$-knot by $T_N(z)$. More precisely, if
the $\cN$-arc components of $T$ are $a_1, \dots, a_k$ and the $\cN$-knot components  are $\al_1, \dots, \al_l$, then

\be\lbl{eq.def1}
\hPhi_N(T) = \sum_{0\le j_1, \dots, j_l\le N} c_{j_1} \dots c_{j_l}  a_1^{(N)}
 \cup \dots \cup  a_k^{(N)} \cup \, \al_1^{(j_1)} \cup \cdots \cup \al_l^{(j_l)} \ \in \  \LMN.
 \ee

\begin{theorem}\lbl{t.ChebyshevFrobenius}
Let $(M,\cN)$ be a marked 3-manifold and  $\xi$ be a complex root of unity. Let $N:=\text{ord}(\xi^4)$ and $\ve:=\xi^{N^2}$. 

There exists
a unique $\BC$-linear map $\Phi_\xi: \cS_\ve(M,\cN) \to \cS_\xi(M,\cN)$ such that if 
$x\in \cS_\ve(M,\cN)$ is presented by an $\cN$-tangle $T$ then 
$
\Phi_\xi(x) = [\hPhi_{N}(T)]_\xi
$
in $ \Sx\MN$. 

In other words, the map $\hPhi_N: \cT\MN \to \cT\MN$ descends to a well-defined map $\Phi_\xi: \cS_\ve(M,\cN) \to \cS_\xi(M,\cN)$, meaning that the following diagram commutes. 

\[
\begin{tikzcd}\lbl{dia.Phixi}
\LMN  \arrow[d,"{[\,\cdot\,]_\ve}"] \arrow[r,"\hPhi_N"] & \LMN \arrow[d,"{[\, \cdot \, ]_\xi}"] \\
\cS_\ve\MN  \arrow[r,"\Phi_\xi"]& \cS_\xi\MN
\end{tikzcd}
\]
\end{theorem}
 Note that if $\ord(\xi^4)=N$ and $\ve=\xi^{N^2}$, then $\ve\in \{\pm 1, \pm i\}$. If $\cN=\emptyset$, then the skein module $\Se\MN$ with $\ve\in \{\pm 1, \pm i\}$
has an interpretation in terms of classical objects and is closely related to the $SL_2$-character variety, see \cite{Turaev,Bullock,PS,Sikora,Marche}.

We call $\Phi_\xi$ the {\em Chebyshev-Frobenius homomorphism}. As mentioned, for the case when $\cN=\emptyset$ (where there are no arc components), Theorem \ref{t.ChebyshevFrobenius} was proven in \cite{BW1} with the help of the quantum trace map, and was reproven in \cite{Le3} using elementary skein methods. 
 We will prove Theorem \ref{t.ChebyshevFrobenius} in Subsection \ref{sec.mainthm}, using a result on skein algebras of triangulable marked surfaces discussed below, which is also of independent interest.

\subsection{Independence of triangulation problem}

Let $\D$ be a triangulation of a (necessarily totally) marked surface $\SM$. 
Suppose $N$ is a positive integer and $\xi \in \BC^\times$ not necessarily a root of unity.  For now we do not require $N=\ord(\xi^4)$.  Let $\ve=\xi^{N^2}$.

By Proposition~\ref{r.Frobenius}, we have a $\BC$-algebra embedding (the Frobenius homomorphism):
\begin{align}
F_N: \sX_{\ve}(\D) \to \sX_\xi(\D),\quad
\lbl{eq.defPhi1}
F_N(a)& = a^N\quad \text{for all} \ a \in \D.
\end{align}

Consider the embedding $\vpD: \cS_\xi \SM\embed \sX_\xi(\D)$ as a coordinate map depending on a triangulation. If we try to define a function on $\cS_\xi\SM$ using the coordinates, then we have to ask if the function is well-defined, i.e. it does not depend on the chosen coordinate system. Let us look at this problem for the Frobenius homomorphism. 

Identify $\cS_\nu \SM$ as a subset of $\sX_\nu(\D)$ via $\vpD$, for $\nu=\ve,\xi$. We investigate when a dashed arrow exists in the following diagram that makes it commute.
\be \lbl{dia.sx}
\begin{tikzcd}
\cS_\ve \SM  \arrow[hookrightarrow]{r} \arrow[dashed,"?"]{d}& \sX_\ve(\D)  \arrow[d, "F_N"] \\
\cS_\xi \SM \arrow[hookrightarrow]{r} & \sX_\xi(\D)
\end{tikzcd}
\ee

We answer the following questions about $F_N$:

A. For what  $\xi \in \Cx$ and $N\in \BN$ does   $F_N$  restrict to a map from $\cS_{\ve}\SM$ to $\cS_{\xi}\SM$ and the restriction does not depend on the triangulation $\D$?

B.  If $F_N$ can restrict to such a map,  can one define the restriction of $F_N$ onto $\cS_{\ve}\SM$ in an intrinsic way, not referring to any triangulation $\D$?

The answers are given in the following two theorems.

Question A is answered by the following theorem.
\begin{theorem} \lbl{thm.A} Suppose $\xi\in \Cx$ and $N\ge 2$ and $\ve=\xi^{N^2}$. Assume that $\SM$ has at least two different triangulations.  
If $F_N: \sX_{\ve}(\D) \to \sX_\xi(\D)$
restricts to a map $\cS_\ve\SM \to \cS_\xi\SM$ for all triangulations $\D$ and the restriction does not depend on the triangulations, then $\xi$ is a root of unity and $N=\ord(\xi^4)$.
\end{theorem}

Question B is answered by the following converse to Theorem \ref{thm.A}.
\begin{theorem}
\lbl{thm.surface}
Suppose $\SM$ is a triangulable surface and $\xi$ is a root of unity. Let  $N =\ord(\xi^4)$ and $\ve=\xi^{N^2}$. Choose a triangulation $\D$  of $\SM$.  

(a) The map $F_N$ restricts to a $\BC$-algebra homomorphism  $F_\xi:\cS_\ve\SM \to \cS_\xi\SM$ which does not depend on the triangulation $\D$.

(b) If $a$ is a $\cP$-arc, then $F_\xi(a) = a^N$, and if $\al$ is a $\cP$-knot, then $F_\xi(\al)= T_N(\al)$.
\end{theorem}

\no{

\begin{theorem}
Suppose $\SM$ is a triangulable surface and $\xi$ is a complex root of unity. Let  $N =\ord(\xi^4)$ and $\ve=\xi^{N^2}$. Choose a triangulation $\D$  of $\SM$.  

(a) The map $F_N$ restricts to a $\BC$-algebra homomorphism  $F_\xi:\cS_\ve\SM \to \cS_\xi\SM$ which does not depend on the triangulation $\D$.

(c) If $a$ is a $\cP$-arc, then $F_\xi(a) = a^N$, and if $\al$ is a $\cP$-knot, then $F_\xi(\al)= T_N(\al)$.
\end{theorem}
We also have the following converse to Theorem \ref{thm.surface}(a), answering Question A above.
\begin{theorem} Suppose $\xi\in \Cx$ and $N\ge 2$ and $\ve=\xi^{N^2}$. Assume that $\SM$ has at least two different triangulations.  
If $F_N: \sX_{\ve}(\D) \to \sX_\xi(\D)$
restricts to a map $\cS_\ve\SM \to \cS_\xi\SM$ for all triangulations $\D$ and the restriction does not depend on the triangulations, then $\xi$ is a root of unity and $N=\ord(\xi^4)$.
\end{theorem}
}

\def\ZxD{\CZ_\xi(\D)}

\def\ZeD{\CZ_\ve(\D)}
\def\XxD{\sX_\xi(\D)}
\def\XzD{\sX_\xi(\D)}
\def\XeD{\sX_\ve(\D)}
\def\tXxD{\widetilde \sX_\xi(\D)}
\def\tXeD{\widetilde\sX_\ve(\D)}
\def\tXzD{\widetilde \sX_\xi(\D)}
\def\tZxiD{\widetilde {\CZ}_\xi(\D)}
\def\tXxiD{\widetilde {\sX}_\xi(\D)}
\def\tPhi{\tilde\Phi}
\def\tSx{\widetilde \cS_\xi}
\def\tSz{\widetilde \cS_\xi}
\def\tSe{\widetilde\cS_\ve}

\def\tX{\widetilde \sX}
\def\tS{\widetilde \cS}

We prove Theorem \ref{thm.A} in Subsection \ref{sec.triind} and Theorem  \ref{thm.surface} in Subsection \ref{sec.trisurfacethm}.
\subsection{Division algebra}\lbl{sec.triind} Assume $\SM$ is triangulable, $\xi \in \Cx$, and $N\in \BN$. Choose a triangulation $\D$ of $\SM$.
Let $\tXeD$ and $\tXzD$  be the division algebras of $\XeD$ and $\XzD$,  respectively. The $\BC$-algebra embedding $F_N:\XeD\to \XzD$ extends to a $\BC$-algebra embedding $$
\tF_N: \tXeD \to \tXzD.
$$ 

For $\nu=\ve,\xi$ let $\tS_\nu\SM$
be the division algebra of   $\cS_\nu\SM$. By Theorem \ref{r.torus} the embedding 
$\vpD : \cS_\nu\SM \embed \sX_\nu(\D)$ induces an isomorphism $\tvpD : \tcS_\nu\SM \overset \cong \longrightarrow \tX_\nu(\D)$. Diagram~\eqref{dia.sx} becomes
\be \notag
\begin{tikzcd}
\tS_\ve \SM  \arrow[rightarrow,"\cong"]{r}[swap]{\tvpD} 
& \tX_\ve(\D)  \arrow[d, "\tF_N"] \\
\tS_\xi \SM \arrow[rightarrow,"\cong"]{r}[swap]{\tvpD} & \tX_\xi(\D)
\end{tikzcd}
\ee
By pulling back $\tF_N$ via $\tvpD$, we get a $\BC$-algebra embedding

\be
\tF_{N,\D}: \tSe\SM \to  \tSz\SM,
\ee
which a priori depends on the $\cP$-triangulation $\D$.

\begin{proposition}\lbl{prop.two} Let $\SM$ be a triangulable marked surface, $\xi\in \Cx$, and $N\in \BN$. 

(a) If $\xi$ is a root of unity and 
 $N:=\text{ord}(\xi^4)$, 
then $\tF_{N,\D}$   does not depend on the triangulation $\D$.

(b)  Suppose $\SM$ has at least 2 different triangulations and $N \ge 2$. Then $\tF_{N,\D}$   does not depend on $\D$ if and only if $\xi$ is a root of unity and $N=\ord(\xi^4)$.

\end{proposition}
\begin{remark} A totally marked surface $\SM$ has at least 2  triangulations if and only if it is not a disk with less than 4 marked points.
\end{remark}

\begin{proof} As (a) is a consequence of (b), let us prove (b).

By Proposition \ref{r.42}, the map $\tF_{N,\D}$   does not depend on $\cP$-triangulations $\D$ if and only if the diagram
\be 
\begin{tikzcd}
\tiX_\ve(\D) \arrow[r,"\Theta_{\D,\D'}" ] \arrow[d, "\tF_N" ]& \tiX_\ve(\D')  \arrow[d, "\tF_N" ] \\
\tiX_\xi(\D) \arrow[r,"\Theta_{\D,\D'}" ] & \tiX_\xi(\D')
\lbl{dia.xx}
\end{tikzcd}
\ee
is commutative for any two $\cP$-triangulations $\D,\D'$. Since any two $\cP$-triangulations are related by a sequence of flips, in \eqref{dia.xx} we can assume that   $\D'$ is obtained from $\D$ by a flip at an edge $a\in \D$, with  the notation as given in Figure \ref{fig:mutation}. Then  $\D'= \D\cup \{a^*\} \setminus \{a\}$.
The commutativity of \eqref{dia.xx} is equivalent to
\be 
\lbl{eq.aa3}
(\tF_N \circ \Theta_{\D,\D'})(x) = (\Theta_{\D,\D'} \circ \tF_N) (x), \quad \text{for all } \ x\in \tiX_\ve(\D).
\ee
 Since  $\D$ weakly generates the algebra $\tiX_\xi(\D)$, it is enough to show \eqref{eq.aa3} for $x\in \D$.

If $x \in \D \setminus \{ a\}$, then by \eqref{eq.aa0} one has $\Theta_{\D,\D'}(x)= x$, and hence we have \eqref{eq.aa3} since both sides are equal to $x^N$ in $\tiX_\ve(\D')$.
Consider the remaining case $x=a$.  By \eqref{eq.aa1}, we know that
\be 
\lbl{eq.aa2}
\Theta_{\D,\D'}(a)  = X + Y, \quad \text{where} \ X=[  b  d  (a^*)^{-1}], \ Y=[  c  e  (a^*)^{-1}].
\ee 
Using the above identity and the definition of $\tF_N$, we  
calculate the left hand side of \eqref{eq.aa3}:
\be  \lbl{eq.l0}
(\tF_N \circ \Theta_{\D,\D'})(a) =  \tF_N\left( [  b  d  (a^*)^{-1}] + [  c  e  (a^*)^{-1}] \right)=  [  b^N  d^N  (a^*)^{-N}] + [  c^N  e^N  (a^*)^{-N}]= X^N + Y^N.
\ee

Now we calculate the right hand side of \eqref{eq.aa3}:
\begin{align}
\lbl{eq.l1}
(\Theta_{\D,\D'} \circ \tF_N)(a) &= \Theta_{\D,\D'} (a^N)  
= \left(\Theta_{\D,\D'} (a)\right)^N     =  (X+Y)^N.
\end{align} 
Comparing \eqref{eq.l0} and \eqref{eq.l1}, we see that \eqref{eq.aa3} holds if and only if 
\be  (X+Y)^N= X^N + Y^N
\lbl{eq.gauss}
\ee From the $q$-commutativity of elements in $\D'$ one can check that $XY = \xi^4 YX$. By the Gauss binomial formula (see eg. \cite{KC}),
$$ (X+Y)^N= X^N + Y^N + \sum_{k=1}^{N-1} \binom Nk_{\xi^4} Y^k X^{N-k},\ \text{
where} \
\binom Nk_{\xi^4}= \prod_{j=1}^k \frac{1- \xi^{4(N-j+1)}}{1-\xi^{4j}}.$$
Note that $Y^k X^{N-k}$ is a power of $\xi$ times a monomial in $b,c,d,e$, and $(a^*)^{-1}$, and these monominals are distinct for $k=0,1,\dots, N$. As monomials (with positive and negative powers) in edges form a $\BC$-basis of $\sX_\xi(\D')$, we see that $(X+Y)^N= X^N + Y^N$ if and only if 
\be  \binom Nk_{\xi^4} =0 \ \text{for all } k = 1,2, \dots, N-1. \lbl{eq.gauss2}.
\ee
It is well-known, and easy to prove, that \eqref{eq.gauss2} holds if and only if $\xi^4$ is a root of unity of order $N$. 
\end{proof}

As the edge $a$ in the proof of Proposition \ref{prop.two} is in $\cS\SM$,  Theorem \ref{thm.A} follows immediately.

\subsection{Frobenius homomorphism $\tF_\xi:=\tF_{N,\D}$}
From now on let $\xi$ be a root of unity, $N=\ord(\xi^4)$, $\ve=\xi^{N^2}$. Suppose $\SM$ is a triangulable marked surface. Since $\tF_{N,\D}$ does not depend on $\D$ and $N=\ord(\xi^4)$, we denote 
$$ \tF_\xi:= \tF_{N,\D}: \tcS_\ve\SM \to \tcS_\xi\SM.$$

\subsection{Arcs in $\SM$}
\begin{proposition} 
\lbl{r.arc5}
Suppose $a\subset \Sigma$ is a $\cP$-arc. Then $\tF_\xi(a)= a^N$. 
\end{proposition}
\begin{proof} Since $a$ is an element of a $\cP$-triangulation $\D$, we have $\tF_\xi(a)= \tF_{N,\D}(a) = a^N$. 
\end{proof}

\def\ttS{\widetilde{\cS}}
\subsection{Functoriality}

\begin{proposition}
\lbl{prop.three} Suppose $\SM$ and  $(\Sigma', \cP')$ are triangulable marked surfaces such that  $\Sigma\subset \Sigma'$ and $\cP\subset \cP'$. For any $\zeta \in \Cx$, the embedding $\iota: \SM \embed (\Sigma',\cP')$ induces a $\BC$-algebra homomorphism $\iota_*: \widetilde \cS_\zeta(\Sigma, \cP) \to \widetilde \cS_\zeta(\Sigma', \cP')$.

Let $\xi\in \Cx$ be a root of unity, $N =\ord(\xi^4)$ and $\ve= \xi^{N^2}$. Then  the following diagram commutes.

\[
\begin{tikzcd}
\ttS_\ve\SM \arrow[r,"\iota_*"] \arrow[d, "\tF_\xi"]& \ttS_\ve(\Sigma',\cP') \arrow[d, "\tF_\xi"] \\
\ttS_\xi\SM \arrow[r,"\iota_*"] & \ttS_\xi(\Sigma',\cP')
\end{tikzcd}
\]

\end{proposition}

\begin{proof} If $a\subset \Sigma$ is a $\cP$-arc, then it is also a $\cP'$-arc in $\Sigma'$. Hence by Proposition \ref{r.arc5}, both $\iota_*\circ \tF_\xi(a)$ and $\tF_\xi\circ \iota_*(a)$ are equal to $a^N$ in $\ttS_\xi(\Sigma',\cP')$. Since for a triangulable marked surface $\SM$, the set of all sums of $\cP$-arcs and their inverses generates $\tXeD=\widetilde\cS_\ve\SM  $, we have the commutativity of the diagram.
\end{proof}

\subsection{Knots in $\SM$} We find an intrinsic definition of $\tF_\xi(\al)$, where $\al$ is a $\cP$-knot. 
\begin{proposition}\lbl{r.knot} Suppose $\SM$ is a triangulable marked surface, $\xi$ is a root of unity, and $N=\ord(\xi^4)$.
 If $\al$ is a $\cP$-knot in $\SM$, then $\tF_\xi(\al) = T_N(\al)$.
\end{proposition}

We break the proof of Proposition \ref{r.knot} into lemmas.
\begin{lemma}\lbl{r.unknot}
(a) Proposition \ref{r.knot} holds if $\al$ is a trivial $\cP$-knot, i.e. $\al$ bounds a disk in $\Sigma$.

(b) If $\xi$ is a root of unity with $\ord(\xi^4)=N$ and $\ve=\xi^{N^2}$, then
\be 
\lbl{eq.vexi}
T_N(-\xi^2-\xi^{-2})= -\ve^2 - \ve^{-2}.
\ee
\end{lemma}
\begin{proof} Let us prove (b) first. The left hand side and the right hand side of \eqref{eq.vexi} are
\begin{align}
\lbl{eq.xx1}
LHS&=T_N(-\xi^2 -\xi^{-2})=(-\xi^2)^N + (-\xi^{-2})^{N}=(-1)^N(\xi^{2N} + \xi^{-2N})\\
RHS&= -\ve^2 - \ve^{-2} = -\xi^{2N^2} -\xi^{-2N^2}. \lbl{eq.xx2}
  \end{align}
  Since $\ord(\xi^4)=N$, either $\ord(\xi^2)=2N$ or $\ord(\xi^2)=N$.
  
  Suppose $\ord(\xi^2)=2N$. Then both  right hand sides of \eqref{eq.xx1} and \eqref{eq.xx2} are equal to $2(-1)^N$, and so they are equal.
  
  Suppose $\ord(\xi^2)=N$. Then $N$ must be odd since otherwise $\ord(\xi^4)=N/2$. 
Then both  right hand sides  of \eqref{eq.xx1} and \eqref{eq.xx2} are equal to $-2$. This completes the proof of (b).

(a) Since $\al$ is a trivial knot, $\al= -\ve^2 - \ve^{-2}$ in $\cS_\ve\SM$ and $\al= -\xi^2 -\xi^{-2}$ in $\cS_\ve\SM$. Hence
\begin{align*}
\tF_\xi(\al) &= \tF_\xi (-\ve^2 - \ve^{-2})= -\ve^2 - \ve^{-2}= T_N(-\xi^2 -\xi^{-2})=
T_N(\al),
\end{align*} 
where the third identity is part (b). Thus $\tF_\xi(\al) = T_N(\al)$.
\end{proof}

\begin{lemma}\lbl{r.markedannulus}  Proposition \ref{r.knot} holds if $\Sigma = \An$, the annulus, and $\cP\subset \partial \An$ consists of two points, one in each connected component of $\partial \An$.

\end{lemma}

\begin{proof} If $\al$ is a trivial $\cP$-knot, then the result follows from  Lemma \ref{r.unknot}. We assume $\al$ is non-trivial. Then $\al$ is the core of the annulus, i.e. $\al$ is a parallel of a boundary component of $\An$.
Let $\D=\{a,b,c,d\}$ be the triangulation of $(\BA,\cP)$ shown in Figure \ref{fig:annulusquasitriangulation}.
\FIGc{annulusquasitriangulation}{Triangulation of $(\BA,\cP)$. $c,d$ are boundary $\cP$-arcs.}{2.2cm}

The Muller algebra $\sX_\xi(\D)$ is a quantum torus with generators $a,b,c,d$ where any two of them commute, except for $a$ and $b$ for which  $ab = \xi^{-2} ba$. 
We calculate $\al$ as an element of $\sX_\xi(\D)$ as follows.

First, calculate $a\al$ by using the skein relation, see Figure \ref{fig:case3eqn2}.
\FIGc{case3eqn2}{Computation of $a\al = \xi b^* + \xi^{-1}b$.}{2cm}

Here $b^*$ is new edge obtained from the flip of $\D$ at $b$ as defined in Figure \ref{fig:mutation}. From Equation (\ref{eq.aa1}) we have that $b^* = [b^{-1}a^2] + [b^{-1}cd]$. Thus,
\begin{align}
\al & = a^{-1} (a\al)= a^{-1} (\xi b^* + \xi ^{-1}b) = a^{-1} (\xi  ([b^{-1}a^2] + [b^{-1}cd] )+ \xi ^{-1}b) \nonumber\\
\lbl{eq.v1}& =   [a^{-1} b^{-1} c d]+[a b^{-1}] +[ a^{-1} b]= X + Y + Y^{-1}.
\end{align}
where $X= [a^{-1} b^{-1} c d]$ and $Y= [a b^{-1}]$. From the commutation relations in $\sX_\xi(\D)$, we get $YX= \xi^4 XY$. 
Since each   of $\{a,b,c,d\}$ is a $\cP$-arc, from Proposition \ref{r.arc5}, we have
\be 
\lbl{eq.v2}
\tF_\xi(\al)= [a^{-N} b^{-N} c^N d^N]+[a^N b^{-N}] +[ a^{-N} b^N]= X^N + Y^N + Y^{-N}.
\ee
Since $\ord(\xi^4)=N$, 
Corollary \ref{c.32} shows that
$$
T_N(\al) = T_N(X + Y + Y^{-1})= X^N + Y^N + Y^{-N},
$$
which is equal to $\tF_\xi(\al)$ by \eqref{eq.v2}. This completes the proof.
\end{proof}

\def\cQ{\mathcal Q}
\begin{lemma} \lbl{r.knot1a} Proposition \ref{r.knot} holds if $\al$ is not 0 in $H_1(\Sigma, \BZ)$. 
\end{lemma}

\begin{proof} 
{\em Claim.} If $\al$ is not 0 in $H_1(\Sigma, \BZ)$, then there exists a properly embedded arc $a \subset \Sigma$ such that $|a \cap \al|=1$. \\
{\em Proof of Claim.} Cutting $\Sigma$ along $\al$ we get a (possibly non-connected) surface $\Sigma'$ whose boundary contains two components $\beta_1, \beta_2$ coming from $\al$. That is, we get $\Sigma'$ from $\Sigma$ by gluing $\beta_1$ to $\beta_2$ via the quotient map $\pr: \Sigma' \to \Sigma$, where $\pr(\beta_1)=\pr(\beta_2)=\al$. Choose a point $p\in \al$ and let $p_i\in \beta_i$ such that $\pr(p_i)=p$ for $i=1,2$.

Suppose first that $\Sigma'$ is connected. For $i=1,2$ choose a properly embedded arc $a_i$ connecting  $p_i\in \beta_i$ and a point in a boundary component of $\Sigma'$ which is not $\beta_1$ nor $\beta_2$. We can further assume that $a_1 \cap a_2=\emptyset$ since if they intersect once then replacing the crossing with either a positive or negative smoothing from the Kauffman skein relation (only one will work) will yield arcs that do not intersect and end at the same points as $a_1,a_2$, and the general case follows from an induction argument. Then $a=\pr(a_1 \cup a_2)$ is an arc such that $|a \cap \al|=1$.

Now suppose $\Sigma'$ has 2 connected components $\Sigma_1$ and $\Sigma_2$, with $\beta_i \subset \Sigma_i$. Since $\al$ is not homologically trivial, each $\Sigma_i$ has a boundary component other than $\beta_i$. For $i=1,2$ choose a properly embedded arc $a_i$ connecting  $p_i\in \beta_i$ and a point in a boundary component of $\Sigma'$ which is not $\beta_i$. Then $a=\pr(a_1 \cup a_2)$ is an arc such that $|a \cap \al|=1$. This completes the proof of the claim.

Let $\cQ = \partial a$ and 
$\cP'=\cP \cup \cQ$. Let $S\subset \Sigma$ be the closure of a tubular neighborhood of $\al \cup a$. Then $S$ is an annulus, and $\cQ$ consists of 2 points, one in each connected component of $\partial S$. Let $\tF_{\xi,(S,\cQ)}$, $\tF_{\xi,\SM}$, and $\tF_{\xi,(\Sigma,\cP')}$ be the map $\tF_\xi$ applicable respectively to the totally marked surfaces $(S,\cQ)$, $\SM$, and $(\Sigma,\cP')$. By the functoriality of the inclusion $(S,\cQ)\subset (\Sigma, \cP')$ (see Proposition \ref{prop.three}) we get the first of the following identities 
$$ \tF_{\xi,(\Sigma,\cP')} (\al) = \tF_{\xi,(S,\cQ)}(\al) = T_N(\al)\ \text{in } \cS_\xi(\Sigma,\cP'),$$
while the second follows from Lemma \ref{r.markedannulus}. The functoriality of the inclusion $ \cP \subset \cP'$  gives
$$ \tF_{\xi,(\Sigma,\cP)} (\al) =  \tF_{\xi,(\Sigma,\cP')} (\al) \ \text{in } \cS_\xi(\Sigma,\cP').$$
It follows that $$ \tF_{\xi,(\Sigma,\cP)} (\al)=  T_N(\al) \ \text{in } \cS_\xi(\Sigma,\cP').$$
Since the natural map $\cS\SM \to \cS(\Sigma, \cP')$ is an embedding (Proposition \ref{r.func}), we also have $ \tF_{\xi,(\Sigma,\cP)} (\al)=  T_N(\al) \ \text{in } \cS_\xi(\Sigma,\cP)$, completing the proof.
\end{proof}

\def\SMp{(\Sigma', \cP')}
Now we proceed to the proof of Proposition \ref{r.knot}.
\begin{proof}[Proof of Proposition \ref{r.knot}]
If $\al\neq 0$ in $H_1(\Sigma,\BZ)$, then the statement follows from Lemma~\ref{r.knot1a}.  Assume  $\al=0$  in $H_1(\Sigma,\BZ)$. The idea we employ is to remove a disk in $\Sigma$ so that $\al$ becomes homologically non-trivial in the new surface, then use the surgery theory developed in Section \ref{sec.surgery}.

 Since $\al=0$ in $H_1(\Sigma,\BZ)$, there is a surface $S\subset \Sigma$ such that $\al=\partial S$. Let $D\subset S$ be a closed disk in the interior of $S$ and  $\beta=\partial D$. Let $\Sigma'$ be obtained from $\Sigma$ by removing the interior of $D$. Fix a point $p\in \beta$ and let $\cP'=\cP \cup \{p\}$. Since $\SM$ is triangulable, $\SMp$ is also triangulable and $(\Sigma', \cP)$ is quasitriangulable.

Choose an arbitrary quasitriangulation $\D'$ of $(\Sigma', \cP)$. Via Proposition \ref{t.holetrick}, by plugging the unmarked boundary component $\beta$ we get a triangulation $\D$ of $\SM$ and a quotient map $\Psi_\zeta: \CZ_\zeta(\D') \to \CZ_\zeta(\D)$ for each $\zeta \in \Cx$, which we will just call $\Psi$ unless there is confusion. 
Since $\D$ is a triangulation, we have $\CZ_\zeta(\D) =\sX_\zeta(\D)$.

For each $\zeta \in \Cx$ we have the inclusions  
$$\CZ_\zeta(\D')\subset \sX_\zeta(\D') \subset \widetilde{\cS}_\zeta(\Sigma', \cP'),$$ where the second one comes from 
$\sX_\zeta(\D') \subset \widetilde{\cS}_\zeta(\Sigma', \cP) \subset \widetilde{\cS}_\zeta(\Sigma', \cP')$.

\noindent {\em Claim 1.} The map $\tF_\xi: \tcS_\ve(\Sigma',\cP') \to \tcS_\xi(\Sigma',\cP')$ restricts to a map from $\CZ_\ve(\D')$ to $\CZ_\xi(\D')$. That is, $\tF_\xi(\CZ_\ve(\D')) \subset \CZ_\xi(\D')$. In other words, there exists a map $F^\beta_\xi$ corresponding to the dashed arrow in the following commutative diagram.
\[
\begin{tikzcd}
\tcS_\ve(\Sigma',\cP') 
\arrow[r,"\tF_\xi"] & \tcS_\xi(\Sigma',\cP')  
\\
\CZ_\ve(\D')\arrow[u, hook] \arrow[r,dashed,"F^\beta_\xi"] & \CZ_\xi(\D') \arrow[u, hook]
\end{tikzcd}
\]
{\em Proof of Claim 1.} Let $a\in \D'$ be the only monogon edge (which must correspond to $\beta$). By definition, the set consisting of \\
(i) elements in $\D'\setminus \{a\}$ and their inverses,  $a$ and $a^*$, and \\
(ii) $\beta$\\  generates the $\BC$-algebra $\CZ_\ve(\D')$. Let us look at each of these generators. If $x$ is an element of type (i) above, then by Proposition \ref{r.arc5}, we have $\tF_\xi(x)= x^N$ which is in $\CZ_\xi(\D')$. Consider the remaining case $x=\beta$. Since the class of $\beta$ in $H_1(\Sigma',\BZ)$ is nontrivial, by Lemma \ref{r.knot1a}, we have
\be \lbl{eq.8aa}
\tF_\xi(\beta)= T_N(\beta)
\ee
 which is also in $\CZ_\xi(\D')$. Claim 1 is proved.

\noindent{\em Claim 2.} The following diagram is commutative.
\be \lbl{eq.cd5}
\begin{tikzcd}
\CZ_\ve(\D') \arrow[d,"\Psi_\ve"] \arrow[r,"F^\beta_\xi"] & \CZ_\xi(\D') \arrow[d,"\Psi_\xi"] \\
\sX_\ve(\D) \arrow[r,"F_N"] & \sX_\xi(\D)
\end{tikzcd}
\ee
{\em Proof of Claim 2.} We have to show that
\be \lbl{eq.8a}
(F_N \circ \Psi)(x) = (\Psi \circ F^\beta_\xi)(x) \quad \text{for all} \ x\in \CZ_\ve(\D').
\ee
It is enough to check the commutativity on the set of generators of $\CZ_\ve(\D')$ described in (i) and (ii) above. If \eqref{eq.8a} holds for $x$ which is invertible, then it holds for $x^{-1}$. Thus it is enough to check \eqref{eq.8a} for $x\in \D' \cup \{ a^*, \beta\}$. 
Assume the notations $a,b,c$ of the edges near $\beta$ are as in Figure \ref{fig:surgery}. 

First assume $x \not \in \{ a,a^*,b,\beta\}$. By \eqref{eq.2a}, we have $\Psi(x)=x$. Hence the left hand side of \eqref{eq.8a} is
$$F_N(\Psi(x))= F_N(x)= x^N.$$
On the other hand, the right hand side of \eqref{eq.8a} is
$$\Psi(F^\beta_\xi(x))=\Psi(\tF_\xi(x)) = \Psi(x^N)= x^N, $$
which proves \eqref{eq.8a} for $x \not \in \{ a,a^*,b,\beta\}$.

Assume $x=a$ or $x=a^*$. By \eqref{eq.2b}, we have $\Psi(x)=0$. Hence the left hand side of \eqref{eq.8a} is 0. On the other hand, the right hand side is
$$\Psi(F^\beta_\xi(x))= \Psi(\tF_\xi(x)) = \Psi(x^N)=0,$$
which proves \eqref{eq.8a} in this case.

Now consider the remaining case $x=\beta$. By \eqref{eq.2b}, we have $\Psi(\beta)= -\ve^2 -\ve^{-2}$. Hence the left hand side of \eqref{eq.8a} is
\be \nonumber
  F_N(\Psi(\beta))= F_N(-\ve^2 -\ve^{-2}) = -\ve^2 -\ve^{-2} = T_N(-\xi^2 -\xi^{-2}),
  \ee
  where the last identity is \eqref{eq.vexi}.
On the other hand, using \eqref{eq.8aa} and the fact that $\Psi$ is a $\BC$-algebra homomorphism, we have
\be 
\nonumber
  \Psi(F^\beta_\xi(x))=\Psi(\tF_\xi(x)) = \Psi(T_N(\beta))= T_N(\Psi(\beta))= T_N(-\xi^2 -\xi^{-2}).
  \ee
 Thus we always have \eqref{eq.8a}. This completes the proof of Claim 2.
 
 Let us continue with the proof of the proposition. Since the class of $\al$ is not 0 in $H_1(\Sigma',\BZ)$, by Lemma \ref{r.knot1a}, we have $F^\beta_\xi(\al)=\tF_\xi(\al)= T_N(\al)$. The commutativity of Diagram \eqref{eq.cd5} and the fact that $\Psi$ is a $\BC$-algebra homomorphism implies that
 \begin{align} F_N(\Psi_\ve(\al ))&= \Psi_\xi(F^\beta_\xi(\al ))= \Psi_\xi(T_N(\al )) \nonumber\\
 &= T_N(\Psi_\xi(\al )).
 \lbl{eq.9p} 
 \end{align}
 Note that $\al $ defines an element in $\cS_\nu(\Sigma', \cP)$ for $\nu=\ve,\xi$. Following the commutativity of Diagram \eqref{d.holesurgery} in Proposition \ref{t.holetrick}, we have that
 \begin{align}
&\Psi_\ve(\al)=\al \in \cS_\ve\SM \subset \sX_\ve(\D), \lbl{eq.alve}\\
&\Psi_\xi(\al)=\al \in \cS_\xi\SM \subset \sX_\xi(\D). \lbl{eq.alxi}
\end{align}

Then we may compute
\begin{align*}
\tF_\xi(\al) & = F_N(\al), \quad \text{by Proposition~\ref{prop.two}} \\
& = F_N(\Psi_\ve(\al)), \quad \text{by \eqref{eq.alve}} \\
& = T_N(\Psi_\xi(\al)), \quad \text{by \eqref{eq.9p}} \\
& = T_N(\al), \quad \text{by \eqref{eq.alxi}},
\end{align*}
completing the proof of Proposition \ref{r.knot}.
\end{proof}

\def\LMN{\cT\MN}
\subsection{Proof of Theorem \ref{thm.surface}}\lbl{sec.trisurfacethm}

\begin{proof} [Proof of Theorem \ref{thm.surface}]

The $\BC$-algebra $\Se\SM$ is generated by $\cP$-arcs and $\cP$-knots. If $a$ is a $\cP$-arc, then by Proposition \ref{r.arc5}, $\tF_\xi(a)= a^N\in \Sx\SM$. If $\al$ is a $\cP$-knot, then by Proposition \ref{r.knot}, $\tF_\xi(\al)= T_N(\al)\in \Sx\SM$. It follows that $\tF_\xi(\Se\SM)\subset \Sx\SM$. Hence  $\tF_\xi$ restricts to a $\BC$-algebra homomorphism $F_\xi: \Se\SM \to \Sx\SM$. Since on $\XD$, $\tF_\xi$ and $F_N$ are the same, $F_\xi$ is the restriction of $F_N$ on $\Se\SM$.
From Proposition \ref{prop.two}, $F_\xi$ does not depend on the triangulation~$\D$. This proves part (a). Part (b) was established in Propositions \ref{r.arc5} and \ref{r.knot}.
\end{proof}


\def\cSe{\cS_\xe}
\subsection{Proof of Theorem \ref{t.ChebyshevFrobenius}}\lbl{sec.mainthm}
\begin{proof}[Proof of Theorem \ref{t.ChebyshevFrobenius}]   Recall that $\hPhi_N: \LMN \to \LMN$ is the $\BC$-linear map defined so that that if 
$T$ is an $\cN$-tangle with arc components $a_1, \dots, a_k$ and  knot components $\al_1,\dots,\al_l$, then
\be\lbl{eq.def00}
\hPhi_N(T) = \sum_{0\le j_1, \dots, j_l\le N} c_{j_1} \dots c_{j_l}  a_1^{(N)}
 \cup \dots \cup  a_k^{(N)} \cup \,\al_1^{(j_1)} \cup \cdots \cup \al_l^{(j_l)} \ 
 \ee 
where $T_N(z) = \sum_{i=0}^N c_i z^i$ is the $N$th Chebyshev polynomial of type 1, see \eqref{eq.Che}. 
To show that $\hPhi_N: \LMN \to \LMN$ descends to a map $\Se\MN \to \Sx\MN$ we have to show that $\hPhi_N(\Rel_\ve) \subset \Rel_\xi$. Let $\hPhi_\xi: \LMN \to \Sx\MN$ be the composition
$$ \hPhi_\xi: \LMN \overset {\hPhi_N} \longrightarrow \LMN \overset{{[\,\cdot\,]_\xi}}{\longrightarrow} \Sx\MN.$$
Then we have to show that $\hPhi_\xi(\Rel_\ve)=0$.
There are 3 types of elements which span $\Rel_\ve$: trivial arc relation elements, trivial knot relation elements, and skein relation elements, and we consider them separately.

(i) Suppose $x$ is a trivial arc relation element (see Figure \ref{fig:trivialarc}).  The $N$ copies $a^{(N)}$ in $\hPhi_\xi(x)$ have $2N$ endpoints, and by 
reordering the height of the endpoints, from  $a^{(N)}$ we can obtain a trivial arc. Hence, the reordering relation (see Figure \ref{fig:boundary}) and the trivial arc relation show that $\hPhi_{\xi}(x) =0$.

(ii) Suppose $x=\ve^2 + \ve^{-2} +\al $ is a trivial loop relation element, where $\al$ is a trivial loop. Each parallel of $\al$ is also a trivial loop, which is equal to $-\xi^2 -\xi^{-2}$ in $\Sx\MN$. Hence $\hPhi_\xi(\al) = T_N(-\xi^2 -\xi^{-2})$, and
$$ \hPhi_{\xi}(x)= \ve^2 + \ve^{-2} + T_N((\al))=  \ve^2 + \ve^{-2} + T_N(-\xi^2 -\xi^{-2})=0,$$
where the last identity is \eqref{eq.vexi}. 

(iii) Suppose $x = T - \ve T_+ - \ve^{-1} T_-$ is a skein relation element. Here 
 $T, T_+, T_-$ are  $\cN$-tangles which are identical outside a ball $B$ in which they look like in Figure \ref{fig:skein1a}.
 
 \FIGc{skein1a}{From left to right: the tangles $T, T_+, T_-$}{1.3cm}
 
 \underline{Case I: the two strands of $T\cap B$ belong to two distinct components of $T$}.
 Let $T_1$ be the component of $T$ containing the overpass strand of $T \cap B$ and $T_2= T \setminus T_1$.
  Let $M'$ be the closure of a small neighborhood of $B \cup T=B\cup T_+=B\cup T_-$.   
 Write
 $\cN':= \cN \cap \partial (M')$. The functoriality 
 of the inclusion $(M', \cN') \to \MN$ implies that it is enough to show
  $\hPhi_\xi(x)=0$ for $(M', \cN')$. Thus now we replace $\MN$ by $(M', \cN')$.

\def\SQ{(\Sigma, \cQ)}
\def\LSQ{\cT\SQ}
\def\hF{\hat F}
 Note that $M'$ is homeomorphic to $\Sigma \times (-1,1)$ where $\Sigma$ is an oriented surface which is the union of the shaded disk of Figure \ref{fig:skein1a} and the ribbons obtained by thickening the tangle $T_+$. As usual, identify $\Sigma $ with $\Sigma \times \{0\}$. Then, all four $\cN'$-tangles $T_1, T_2, T_+, T_-$ are in $\Sigma$ and have vertical framing. Note that $\Sigma$ might be disconnected, but each of its connected components has non-empty boundary. Let $\cP= \cN \cap \Sigma$. Then $\cS_\nu(M', \cN')= \cS_\nu\SM$ for $\nu=\xi,\ve$. Enlarge $\cP$ to a larger set of marked points $\cQ$ such that $(\Sigma, \cQ)$ is triangulable. Since the induced map $\iota_*: \Sx\SM \to \Sx\SQ$ is injective (by Proposition \ref{r.func}), it is enough to show that $\hPhi_\xi(x)=0$ in $\Sx\SQ= \LSQ/\Rel_\xi$. Here $\LSQ:= \cT(\Sigma \times (-1,1), \cQ \times (-1,1))$.
 
  The vector space $\LSQ$ is a $\BC$-algebra, where the product $\al \beta$ of two ($\cQ\times (-1,1)$)-tangles $\al$ and $\beta $ is the result of placing $\al$ on top of $\beta$. The map $\hPhi_\xi :\LSQ \to \Sx\SQ$ is an algebra homomorphism. Recall that for an element $y\in \LSQ$ we denote by $[y]_\nu$ its image under the projection $\LSQ \to \cS_\nu(\Sigma,\mathcal{Q})= \LSQ/\Rel_\nu$ for $\nu=\xi,\ve$.
 
 As $\SQ$ is triangulable, by Theorem \ref{thm.surface} we have the map $F_\xi: \Se\SQ\to \Sx\SQ$. 

Suppose  $y$ is a component of one of $T_1,T_2,T_+,T_-$, then $y$ is either a $\cQ$-knot (in $\Sigma$) or a $\cQ$-arc (in $\Sigma$) whose end points are distinct, with vertical framing in both cases. It follows that $[y^{(k)}]_\xi= [y^k]_\xi$. If $y$ is a knot component then Proposition \ref{r.knot} shows that $F_\xi([y]_\ve)= T_N([y]_\xi)= \hPhi_\xi(y) $.
Each of $T_1,T_2,T_+,T_-$ is the product (in $\LSQ$) of its components as the components are disjoint in $\Sigma$. Hence from the definition of $\hPhi_\xi$, we have
\be
\notag 
\hPhi_{\xi} (T_i)= F_\xi ([T_i]_\ve) \quad \text{for all } \ T_i \in \{ T_1, T_2, T_+, T_-\}.
\ee
As $T= T_1 T_2$ in $\LSQ$, we  have
 \be 
 \notag
 \hPhi_{\xi}(T)=  \hPhi_{\xi}(T_1 T_2) = \hPhi_{\xi}(T_1) \hPhi_{\xi}(T_2) = F_\xi ([T_1]_\ve) F_\xi([T_2]_\ve)= F_\xi([T]_\ve).
 \ee
As $x= T- \ve T_+ - \ve^{-1} T_-$, we also have $\hPhi_\xi(x) = F_\xi([x]_\ve).$ But $[x]_\ve =0$ because $x$ is a skein relation element.
This completes the proof that $\hPhi(x)=0$ in Case I.

\underline{Case II: Both strands of $T \cap B$ belong to the same component of $T$}. We show that this case reduces to the previous case.

Both strands of $T \cap B$ belong to the same component of $T$ means that some pair of non-opposite points of $T \cap \partial B$ are connected by a path in $T\setminus B$ (see Figure \ref{fig:nonopposite}). Assume that the two right hand points of $T \cap \partial D$ are connected by a path in $T\setminus B$. All other cases are similar.
\FIGc{nonopposite}{Four possibilities for non-opposite points being connected by a path in $T\setminus B$.}{2cm}
Then the two strands of $T_+$ in $B$ belong to two different components of $T_+$.
We isotope $T_+$ in $B$ so that its diagram forms a bigon, and calculate $\hPhi_{\xi}(T_+)$ as follows.
 
 \begin{align}  \hPhi_{\xi} \Lplu&=  \hPhi_{\xi} \Lpluss   \qquad \text{by isotopy} \nonumber\\
  &= \ve\,  \hPhi_{\xi} \uplus + \ve^{-1} \hPhi_{\xi} \uminus  \notag\\
 &= \ve (-\ve^{-3})  \hPhi_{\xi}\Lminus + \ve^{-1} \hPhi_{\xi} \LLL   \notag\\
 & =  -\ve^{-2}\hPhi_{\xi} (T_-) + \ve^{-1} \hPhi_{\xi} (T), \lbl{eq.zz1}
 \end{align} 
 where the second equality follows from the skein relation which can be used since the two strands of $T_+$ in the applicable ball belong to different components of $T_+$  (by case I), and the third equality follows from the well-known identity correcting a kink in the skein module:
$$ \negkinka = \ve\negkinkb + \ve^{-1}\negkinkc = (\ve + \ve^{-1}(-\ve^2-\ve^{-2}))\negkinkd = -\ve^{-3}\negkinkd.
$$
The  identity \eqref{eq.zz1} is equivalent to $\hPhi_\xi(x)=0$. 
This completes the proof of the theorem.
\end{proof}

\def\Supp{\text{Supp}}
\def\hatP{{\widehat{\partial}}}

\subsection{Consequence for marked surfaces} Suppose $\SM$ is a marked surface, with no restriction at all. Apply Theorem \ref{t.ChebyshevFrobenius} to $\MN=(\Sigma \times (-1,1), \cP \times (-1,1))$. Note that in this case $\Phi_\xi$ is automatically an algebra homomorphism. Besides, since the set of $\cP$-arcs and $\cP$-knots generate $\Se\SM$ as an algebra,  we get the following corollary.

\begin{proposition}\lbl{prop.five}
Suppose $\SM$ is a  marked surface, $\xi$ is a root of unity, $N=\ord(\xi^4)$, and $\ve = \xi^{N^2}$.  Then there exists a unique $\BC$-algebra homomorphism $\Phi_\xi: \cS_\ve\SM \to \cS_\xi\SM$ such that for $\cP$-arcs $a$ and $\cP$-knots $\al$,
$$
\Phi_\xi(a) = a^N, \ \ \ \ \Phi_\xi(\al) = T_N(\al).
$$
\end{proposition}
\begin{remark}
It follows from uniqueness that in the case where $\SM$ is a triangulable surface, $\Phi_\xi$ is the same as $F_\xi$ obtained in Theorem \ref{thm.surface}.
\end{remark}

\section{Image of $\Phi_\xi$ and (skew-)transparency}\lbl{sec.trans}

In this section we show that the image of the Chebyshev-Frobenius homomorphism $\Phi_\xi$ is either ``transparent'' or ``skew-transparent'' depending on whether $\xi^{2N}=\pm 1$. Our result generalizes known theorems regarding the center of the skein algebra \cite{BW1,Le2,FKL} of an unmarked surface and (skew-)transparent elements in the skein module of an unmarked 3-manifold \cite{Le2}.

We fix the ground ring to be $R=\BC$ throughout this section.
\def\hGamma{\hat{\Gamma}}
\subsection{Center of the skein algebra of an unmarked surface}

Fix a compact oriented surface $\Sigma$ with (possibly empty) boundary. For $\xi \in \Cx$ we write $\cS_\xi := \cS_\xi(\Sigma,\emptyset)$. In this context, and when $\xi$ is a root of unity, the Chebyshev-Frobenius homomorphism $\Phi_\xi: \cS_\ve \to \cS_\xi$ specializes to the Chebyshev homomorphism for the skein algebra of $\Sigma$ given in \cite{BW1}. The image of $\Phi_\xi$ is closely related to the center of $\cS_\xi$.


\begin{theorem}[\cite{FKL}]\lbl{r.bwcenter}
Let $\xi$ be a root of unity, $N= \ord(\xi^4)$, $\ve = \xi^{N^2}$, and $\cH$ the set of boundary components of $\Sigma$. Note that $\xi^{2N}$ is either 1 or $-1$.
Then the center $Z(\cS_\xi)$ of $\cS_\xi$ is given by

\be\lbl{e.surfacecenter}
Z(\cS_\xi) = \left\{ \begin{array}{ll}
\Phi_\xi(\cS_\ve)[\cH] & \text{if }\xi^{2N}=1, \\
\Phi_\xi(\cS_\ve^{\text{ev}})[\cH] & \text{if }\xi^{2N}=-1. \end{array} \right.
\ee
\end{theorem}
Here $\cS_\xi^{\text{ev}}$ is the $\BC$-subspace of $\cS_\xi$ spanned by all 1-dimensional closed submanifolds $L$ of $\Sigma$ such that $\mu(L,\al) \equiv 0 \pmod{2}$ for all knots $\al\subset \Sigma$.
In \cite{BW1}, the right hand side of \eqref{e.surfacecenter} was shown to be a subset of the left hand side using methods of quantum Teichm\"uller theory in the case where $\xi^{2N}=1$. This result was reproven in \cite{Le2} using elementary skein methods. The generalization to $\xi^{2N}=-1$ and the converse inclusion was shown in \cite{FKL}.

\def\LMNT{\mathcal{T}(M \setminus T, \cN \setminus T)}
\subsection{(Skew-)transparency} In the skein module of a 3-manifold, we don't have a product structure, and hence cannot define central elements. Instead we will use the notion of {\em (skew-)transparent elements}, first considered in \cite{Le2}. Throughout this subsection we fix a marked 3-manifold $\MN$.

Suppose $T'$ and $T$ are disjoint $\cN$-tangles. 
Since $\Phi_\xi(T')$ can be presented by a $\BC$-linear combination of $\cN$-tangles in a small neighborhood of $T'$, one can define $\Phi_\xi(T')\cup T$  as an element of $\cS_\xi\MN$, see Subsection \ref{sec.func}.

Suppose $T_1,T_2$, and $T$ are $\cN$-tangles. We say that $T_1$ and $T_2$ are connected by a {\em  single $T$-pass $\cN$-isotopy} if there is a 
 continuous family of $\cN$-tangles $T_t$, $t \in [1,2]$, connecting $T_1$ and $T_2$ such that $T_t$ is transversal to $T$ for all $t \in [1,2]$ and furthermore that $T_t \cap T=\emptyset$ for $t \in [1,2]$ except for a single $s \in (1,2)$ for which $|T_s \cap T|=1$.

\begin{theorem}\lbl{r.mutransparent}
Suppose $\MN$ is a marked 3-manifold, $\xi$ is a root of unity, $N = \ord(\xi^4)$. Note that $\xi^{2N}$ is either 1 or $-1$.

(a) If $\xi^{2N}=1$ then the image of the Chebyshev-Frobenius homomorphism is {\em transparent} in the sense that if $T_1, T_2$ are $\cN$-isotopic $\cN$-tangles disjoint from another $\cN$-tangle $T$, then in $\cS_\xi\MN$ we have
\be 
\Phi_\xi(T) \cup T_1 = \Phi_\xi(T) \cup T_2.
\ee 

(b) If $\xi^{2N}=-1$ then the image of the Chebyshev-Frobenius homomorphism is {\em skew-transparent} in the sense that if $\cN$-tangles $T_1, T_2$ are connected by a single $T$-pass $\cN$-isotopy, where $T$ is another $\cN$-tangle, then in $\cS_\xi\MN$ we have
\be 
\Phi_\xi(T) \cup T_1 = - \Phi_\xi(T) \cup T_2.
\ee 
\end{theorem}

\begin{proof}

(a) and (b) are proven in \cite{Le2} for the case where $\cN = \emptyset$. That is, given an $\emptyset$-tangle $T$, it is shown that $\Phi_\xi(T)$ is (skew-)transparent in $\cS_\xi(M,\emptyset)$ where we necessarily have that all components of $T$ are knots. By functoriality, $\Phi_\xi(T)$ is (skew-)transparent in $\cS_\xi\MN$ as well when all components of $T$ are knots.

We show that $\Phi_\xi(a)$ is (skew-)transparent when $a$ is a $\cN$-arc. Let $T_1,T_2$ be $\cN$-isotopic $\cN$-tangles connected by a single $a$-pass $\cN$-isotopy $T_t$. Consider a neighborhood $U$ consisting of the union of a small tubular neighborhood of $a$ and a small tubular neighborhood of $T_t$. We may assume that the strands of $\Phi_\xi(a)=a^{(N)}$ are contained in the tubular neighborhood of $a$, and furthermore that $T_t$ is a single $a_i$-pass $\cN$-isotopy for each component $a_i$ of $a^{(N)}$. Write $\cQ = \cN \cap U$. Then both $a^{(N)} \cup T_1$ and $a^{(N)} \cup T_2$ are $(U,\cQ)$-tangles and we apply the skein relation and trivial arc relation inductively in $\cS_\xi(U,\cQ)$ to derive the equations in Figure \ref{fig:center4}.

\FIGc{center4}{Resolving crossings between $a^{(N)}$ and $T_1,T_2$ in $\cS_\xi(U,\cQ)$.}{3.5cm}

By functoriality, the computation in $\cS_\xi(U,\cQ)$ is true in $\cS_\xi\MN$ as well. We see from Figure \ref{fig:center4} that $\Phi_\xi(a)$ is transparent if and only if $\xi^{N}=\xi^{-N}$, i.e. that $\xi^{2N}=1$. We also see that $\Phi_\xi(a)$ is skew-transparent if and only if $\xi^{N}=-\xi^{-N}$, i.e. that $\xi^{2N}=-1$.
\end{proof}

\section{Center of the skein algebra for $q$ not a root of unity}\lbl{sec.reducedskein}
\def\siD{\overline{\D}_\inn}
\def\siP{\overline{P}_\inn}
\def\reS{\overline{\cS}}
\def\reX{\overline{\sX}}
\def\revpD{\overline{\vpD}}
\def\reNf{\overline{\fN}}
\def\renf{\overline{\mathfrak{n}}}
\def\oP{\overline{P}}
Throughout this section $R$ is a commutative domain with a distinguished invertible element $q^{1/2}$ and $\SM$ is a marked surface ($R$ is no longer required to be Noetherian in this section). We write $\cS$ for the skein algebra $\cS\SM$ defined over $R$. We will calculate the center of $\cS\SM$ for the case when $q$ is not a root of unity.

\subsection{Center of the skein algebra} Let $\Hu$ denote the set of all unmarked components in $\pS$ and $\Hm$ the set of all marked components.  If $\beta\in \Hu$ let $z_\beta=\beta$ as an element of $\cS$. If $\beta\in \Hm$ let 
$$ z_\beta = \left [ \prod a \right] \in \cS,$$
where the product is over all boundary $\cP$-arcs in $\beta$.

\begin{theorem}
\lbl{thm.center1}
Suppose $\SM$ is a  marked surface. Assume that $q$ is not a root of unity. Then the center $Z(\cS\SM)$ of $\cS\SM$ is the $R$-subalgebra generated by $\{z_\beta \mid  \beta \in\Hu\cup \Hm\}$.
\end{theorem}
\begin{proof}
It is easy to verify Theorem \ref{thm.center1} for the few  cases of  non quasi-triangulable surfaces. We will from now on assume that $\SM$ is quasitriangulable, and fix a quasitriangulation $\D$ of $\SM$.

We write $Z(A)$ to denote the center of an algebra $A$. 
\begin{lemma}\lbl{r.weakcenters}
Let $A \subset B$ be $R$-algebras such that $A$ weakly generates $B$. Then $Z(A) \subset  Z(B)$.
\end{lemma}

\begin{proof}
 Let $x \in Z(A)$. Then $x$ commutes with elements of $A$, and hence with their inverses in $B$ if the inverses exist.  So $x \in Z(B)$.
\end{proof}

 Since $\sX_+(\D)$ weakly generates $\XD$, and $\sX_+(\D) \subset \cS \subset \XD$ by
 Theorem \ref{r.torus}, we have 
\begin{corollary}\lbl{l.centerinclusions}
One has   
$Z(\cS) \subset Z(\XD) $. Consequently $Z(\cS) =  Z(\XD) \cap  \cS$.
\end{corollary}

\begin{lemma}\lbl{l.zbetacentral} For all $\beta \in\Hu\cup \Hm$, one has
$z_\beta \in Z(\cS) \subset Z(\sX(\D))$.
\end{lemma}

\begin{proof} It is clear that $z_\beta \in Z(\cS)$ if $\beta\in \Hu$.
Let $\beta \in \Hm$.
Any  $\cP$-knot $\al$ can be isotoped away from the boundary, and therefore $z_\beta \al = \al z_\beta$.
Let $a \in \cS$ be a $\cP$-arc. If $a$ does not end at some $p \in \beta \cap \cP$ then $az_\beta=z_\beta a$ is immediate. Assume that $a$ has an end at $p \in \beta \cap \cP$. $\cP$-isotope $a$ so that its interior does not intersect $\pS$. Then in the support of $z_\beta$ there is one strand clockwise to $a$ at $p$ and one counterclockwise. Therefore $az_\beta = z_\beta a$ by the reordering relation in Figure \ref{fig:boundary}.
Since $\cS$ is generated as an $R$-algebra by $\cP$-arcs and $\cP$-knots, we have $z_\beta \in Z(\cS)$.
\end{proof}

 For each $\beta \in \Hm$, we define $\bk_\beta \in \BZ^\D$ so that $z_\beta= X^{\bk_\beta}$. In other words,
\be\lbl{e.bkbeta}
\bk_\beta(a) = \left\{ \begin{array}{ll}
1 & \text{if }a \subset \beta, \\
0 & \text{otherwise.} \end{array} \right.
\ee

\def\XpD{\sX_+(\D)} 
We write $P$ for the vertex matrix of $\D$ (see Subsection \ref{sec.vmatrix}).

\begin{lemma}\lbl{l.xcenter}
One has   $Z(\sX(\D)) =  R[\Hu][X^\bk \mid  \bk \in \ker P]$.

\end{lemma}

\begin{proof}

Recall that $\sX(\D)$ is a $\BZ^\D$-graded algebra given by
$$
\sX(\D) = \bigoplus_{\bk \in \BZ^\D} R[\Hu] \cdot X^\bk.
$$
The center of a graded algebra is the direct sum of the centers of the homogeneous parts. Hence
$$
Z(\sX(\D)) = \bigoplus_{\bk \in \BZ^\D: X^\bk \text{ is central}} R[\Hu] \cdot X^\bk.
$$
By the commutation relation \eqref{e.normalizedtorus} we have $X^\bk X^\bl = q^{\langle \bk, \bl \rangle_P} X^\bl X^\bk$. Thus $X^\bk$ is central if and only if $q^{\langle \bk, \bl \rangle_P} = 1$ for all $\bl \in \BZ^\D$. Since $q$ is not a root of unity, this is true if and only if $\langle \bk, \bl \rangle_P = 0$ for all $\bl\in \BZ^\D$. Equivalently, $\bk \in \ker P$. This proves the lemma.
\end{proof}

\begin{lemma}\lbl{r.kerPZm}
The kernel $\ker P$ is the free $\BZ$-module with basis $\{\bk_\beta \mid  \beta \in \Hm\}$.
\end{lemma}

\begin{proof}  Let $\text{Null}(P)$ be the nullity of $P$.

Lemmas \ref{l.zbetacentral} and \ref{l.xcenter} imply that $\bk_\beta \in \ker P$ for each $\beta \in \Hm$. Since the $\bk_\beta$'s, as functions from $\D$ to $\BZ$, have pairwise disjoint supports, the set $\{\bk_\beta \mid  \beta \in \Hm\}$ is $\BQ$-linear independent. In particular, 
\be 
\lbl{eq.hh}
\text{Null}(P) \ge |\Hm|.
\ee

{\em Claim 1.} Assume that $\beta\in \Hu$ is an unmarked boundary component. Choose a point $p_\beta \in \beta$ and let $\cP'=\cP \cup \{p_\beta\}$. Then let $\D'$ be an extension of $\D$ to a $\cP'$-quasitriangulation as depicted in Figure \ref{fig:holepoint} (this guarantees that $\D \subset \D'$), and $P'$ the associated vertex matrix. Then $\text{Null}(P') \geq \text{Null}(P) + 1$.

{\em Proof of Claim 1.} Consider $\BZ^\D \subset \BZ^{\D'}$ via extension by zero. Choose a $\BZ$-basis $B$ of $\ker P$. Then because the $\D \times \D$ submatrix of $P'$ equals $P$, one has $B\subset \ker P'$. Let $\bk_\beta \in \BZ^{\D'}$ be as given in \eqref{e.bkbeta} with $\D$ replaced by $\D'$. Then $\bk_\beta \in \ker P'$.
Since $\bk_\beta$ does not have support in $\D$, $\bk_\beta$ is $\BZ$-linearly independent of $B$. Therefore $\text{Null}(P')$ must be at least 1 greater than $\text{Null}(P)$. This completes the proof of Claim 1.

{\em Claim 2.} One has $\text{Null}(P) = |\Hm|$.

{\em Proof of Claim 2.} 
By \cite[Lemma 4.4(b)]{Le3}, the claim is true if  $\SM$ is totally marked, i.e. if $\Hu=\emptyset$. 

Suppose $|\Hu|=k$. By sequentially adding marked points to unmarked components in $\Hu$ and extending the triangulation as in Claim 1, we get a totally marked surface $(\Sigma, \cP^{(k)})$ with a new vertex matrix $P^{(k)}$. From Claim 1 we have $\text{Null}(P^{(k)}) \ge \text{Null}(P) + k$. On the other hand since $(\Sigma, \cP^{(k)})$ is totally marked and having $|\Hm|+k$ boundary components, we have $\text{Null}(P^{(k)})= |\Hm|+k$. It follows that $|\Hm|   \ge \text{Null}(P)$. Together with \eqref{eq.hh} this shows $\text{Null}(P) = |\Hm|$, completing the proof of Claim 2.

Claim 2 and the fact that  $\{\bk_\beta \mid  \beta \in \Hm\}$ is a $\BQ$-linear independent subset of $\ker P$ shows that  $\{\bk_\beta \mid  \beta \in \Hm\}$ is a $\BQ$-basis of $\ker P$. Let us show that $\{\bk_\beta \mid  \beta \in \Hm\}$ is a $\BZ$-basis of $\ker P$.

Let $x \in \BZ^\D$ be in $\ker P$. Since $\{\bk_\beta \mid  \beta \in \Hm\}$ is $\BQ$-basis of $\ker P$, we have
$$
x = \sum_{\beta \in \Hm} c_\beta \bk_\beta, \quad c_\beta \in \BQ.
$$
Since $\bk_\beta$'s are functions from $\D$ to $\BZ$, have pairwise disjoint supports, and $x:\D \to \BZ$ has integer values, each $c_\beta$ must be an integer. Hence $x$ is a $\BZ$-linear combination of $\{\bk_\beta \mid  \beta \in \Hm\}$. This shows that $\{\bk_\beta \mid  \beta \in \Hm\}$ is a $\BZ$-basis of $\ker P$.
\end{proof}
Lemmas \ref{l.xcenter} and \ref{r.kerPZm} show that $Z(\XD)= R[\Hu][X^{\bk_\beta} \mid \beta \in \Hm]$, which is a subset of $\cS$.  By Corollary~\ref{l.centerinclusions}, we also have $Z(\cS)= Z(\XD) \cap \cS= R[z_\beta | \beta \in \Hu \cup \Hm]$, completing the proof.
\end{proof}


\end{document}